\documentclass[a4paper;11pt]{amsart}
\usepackage{mathrsfs}\usepackage{graphicx}
 \usepackage[arrow,matrix]{xy}
\usepackage{amsmath,amssymb,amscd,bbm,amsthm,mathrsfs, bm}
\usepackage{color,xcolor}
\usepackage{graphicx}
\usepackage{manfnt}
\usepackage[colorlinks, citecolor=blue,pagebackref,hypertexnames=false]{hyperref}

\newtheorem{thm}{Theorem}[section]
\newtheorem{lem}{Lemma}[section]
\newtheorem{defn}{Definition}[section]
\newtheorem{cor}{Corollary}[section]
\newtheorem{prop}{Proposition}[section]
\newtheorem{rem}{Remark}[section]

\theoremstyle{definition}

\begin{document}
\numberwithin{equation}{section}

\textheight=218mm

\newcommand{\set}[1]{\mathfrak{#1}}
\newcommand{\vrect}{\set{V}}
\newcommand{\Stack}{\set{S}}
\newcommand{\tile}{\set{T}}
\newcommand{\cent}{\operatorname{cent}}
\newcommand{\cdim}{n}
\newcommand{\lset}{\left\lbrace}
\newcommand{\rset}{\right\rbrace}

\allowdisplaybreaks

\title[Cauchy--Szeg\H{o} projection]
{Fundamental properties of Cauchy--Szeg\H{o} projection on quaternionic Siegel upper half space and applications
}
\author{Der-Chen Chang, Xuan Thinh Duong, Ji Li, Wei Wang and Qingyan Wu}

\address{Der-Chen Chang,
Department of Mathematics and Department of Computer Science,
Georgetown University, Washington D.C. 20057, USA
and
Graduate Institute of Business Administration, College of Management,
Fu Jen Catholic University, New Taipei City 242, Taiwan, ROC.}
\email{chang@georgetown.edu}

\address{Xuan Thinh Duong, Department of Mathematics, Macquarie University, NSW, 2109, Australia.}
\email{xuan.duong@mq.edu.au}

\address{Ji Li, Department of Mathematics, Macquarie University, NSW, 2109, Australia.}
\email{ji.li@mq.edu.au}

\address{Wei Wang, Department of Mathematics, Zhejiang University, Zhejiang 310027, China.}
\email{wwang@zju.edu.cn}

\address{Qingyan Wu, Department of Mathematics\\
         Linyi University\\
         Shandong, 276005, China
         }
\email{qingyanwu@gmail.com}


\begin{abstract}
We investigate the Cauchy--Szeg\H{o} projection for quaternionic Siegel upper half space to obtain the pointwise (higher order) regularity estimates for Cauchy--Szeg\H{o} kernel and prove that the Cauchy--Szeg\H{o} kernel is non-zero everywhere, which further yields a non-degenerated pointwise lower bound.
As applications, we prove the uniform boundedness of Cauchy--Szeg\H{o} projection on every atom on the quaternionic Heisenberg group, which is used to give an atomic decomposition of regular Hardy space $ H^p$ on quaternionic Siegel upper half space  for $2/3<p\leq1$. Moreover, we establish the characterisation of singular values  of the commutator of
Cauchy--Szeg\H{o} projection based on the kernel estimates and on the recent new approach by Fan--Lacey--Li. The quaternionic structure (lack of commutativity) is encoded in the symmetry groups of  regular functions and the associated partial differential equations.
\end{abstract}

\hskip-1cm
\maketitle
\section{Introduction }

\subsection{Background}
The theory of slice regular functions of a quaternionic variable has been studied intensively (cf. e.g. \cite{ACMM,CGSS,DRSY,GS,Q} and references therein)
and applied successfully to the study of quaternionic  closed operators, quaternionic function spaces and operators on them,
e.g. quaternionic  slice   Hardy space, quaternionic de Branges space and  quaternionic Hankel operator  etc.
(cf. e.g.  \cite{ACQS,ACS,ACS21,CGSS,Sa} and references therein). Meanwhile,
 quaternionic  analysis of several variables has been developed substantially in the  last three decades.   The quaternionic  counterpart of the $\overline{\partial}$-complex  and the $k$-Cauchy--Fueter complex  is  known explicitly now. And the analysis of  these complexes and their operators plays  an important role in studying quaternionic  regular functions (cf. e.g. \cite{adams2,Ba,bS,bures,CSS,Wang,wang19} and references therein). There are also important results along this direction on quaternionic Monge--Amp\`ere equations,
 quaternionic pluripotential theories and quaternionic Calabi--Yau problem (cf. e.g. \cite{alesker1,alesker4,alesker3,AS17,KS,alesker6,wan-wang} and references therein).
 The latter one has its origin in supersymmetric string theory.

 In this paper, we obtain a higher order regularity estimate and a sharp pointwise lower bound for the Cauchy--Szeg\H{o}  kernel on quaternionic Siegel upper half space. The Cauchy--Szeg\H{o} projection plays an important role in complex analysis \cite{FEF,NRSW,LS}.   It is also a key tool connecting the regular function in quaternionic Siegel upper half space to its boundary value. Based on this, {we further establish an atomic decomposition of  quaternionic regular Hardy space over the quaternionic Siegel upper half space and the characterisation of singular values of the commutator of Cauchy--Szeg\H{o} projection. Because of lack of commutativity in quaternionic case, 
most of the known methods from complex analysis are not available. 
In fact, in many cases even the statements of analogous results are different.}
 Nevertheless, we treat a quaternionic regular function  as a vector-valued function  satisfying certain partial differential equations, hence certain techniques in harmonic analysis and PDEs can be applied. From this point of view, the quaternionic structure is encoded in the symmetry groups of  regular functions, the spaces and PDEs, which are different from the complex ones.

Let $\mathbb H^{n }$ be the $n$-dimensional quaternion
space, which is the collection of $n$-tuples $(q_1, \ldots  , q_n), q_l  \in \mathbb{H}$. We write
$q_l = x_{4l-3} + x_{4l-2}\mathbf i + x_{4l-1}\mathbf j + x_{4l}\mathbf  k$, $l = 1, \ldots  , n$.
The quaternionic  Siegel upper half space  is
$\mathcal U :=\left\{q=(q_1,\cdots, q_n)=(q_1,q')\in\mathbb H^n\mid \operatorname{Re} q_1>|q'|^2\right\}$
, whose boundary
\begin{equation}\label{eq:boundary}
 \partial \mathcal U :=\{(q_1,q')\in\mathbb H^n\mid\rho:= {\rm Re}\,q_1-|q'|^2=0\}
\end{equation}
 is a quadratic hypersurface.
  A $C^1$-smooth function $f = f_1 +\mathbf{i }f_2 +\mathbf{j }f_3 +\mathbf{k} f_4 : \mathcal U\rightarrow \mathbb{ H} $ is
called {\it (left) regular} on $\mathcal U  $ if it satisfies the Cauchy--Fueter equations
$\overline{ \partial}_{q_l} f (q) = 0,~  l = 1, \ldots , n$, for any $q\in \mathcal U  $,
where
\begin{equation}\label{eq:CF}
   \overline{\partial}_{q_l}:=\partial_{  x_{4l-3}} + \partial_{ x_{4l-2}} \mathbf i + \partial_{ x_{4l-1}} \mathbf j + \partial_{ x_{4l}} \mathbf  k.
\end{equation}

 The Hardy space $ H^p(\mathcal U )$ consists of all regular functions $F$ on $\mathcal U $, for which
$$\|F\|_{  H^p(\mathcal U )}:=\left(\sup_{\varepsilon>0}\int_{\partial \mathcal U }|F_\varepsilon(q)|^pd\beta(q)\right)^{1\over p}<\infty,$$
 where $F_\varepsilon$ is for its ``vertical translate'', i.e. $F_\varepsilon (q) = F(q + \varepsilon \bf{e_1})$,
where $ {\bf e_1} = (1, 0, 0, \ldots  , 0)$.
 The Cauchy--Szeg\H{o} projection is the  operator from   $L^2(\partial\mathcal U )$ to  $ H^2(\mathcal U )$
  satisfying the following reproducing
formula:
\begin{equation*}\label{eq:Szego0}
    F(q) =
\int_{\partial\mathcal{U} }
S(q, p)F^b(p) d\beta (p),  \qquad q\in \mathcal{U} ,
 \end{equation*}whenever $F \in   H^2(\mathcal{U} )$ with the boundary value $F^b$  on $\partial \mathcal{U} $, where $S(q,p)$ is the Cauchy--Szeg\H{o} kernel:
\begin{align}\label{cauchy-szego}
S(q,p)=s\Big(q_1+\overline p_1-2\sum_{k=2}^n\overline p_kq_k\Big)
\end{align}
for $p= (p_1,\cdots, p_n)\in\mathcal U$, $q=(q_1,\cdots, q_n)\in\mathcal U$, and
\begin{align}\label{s}
s(\sigma)=c_{n-1}{\partial^{2(n-1)}\over \partial x_1^{2(n-1)}}{{\overline \sigma}\over |\sigma|^4},\quad
\sigma=x_1+x_2{\bf i}+x_3{\bf j}+x_4{\bf k}\in\mathbb H
\end{align}
with the real constant $c_{n-1}$ depending only on $n$  (\cite[{Theorem A}]{CMW}).

Note that  the boundary $ \partial \mathcal U$ can be identified with the quaternionic Heisenberg group
 $\mathscr H^{n-1}$, and  the Siegel upper half space $\mathcal U$ can be identified with $\mathscr U:=\mathbb{R}_+\times\mathscr H^{n-1}$ by a quadratic diffeomorphism (\ref{eq:pi}). A regular  function   on $\mathscr U $ can be characterized by  Proposition \ref{prop:regular-equiv}. 
 Details of these notations will be introduced in Section 2.
For $0<p<\infty$, the Hardy space $  H^p(\mathscr U )$ consists of all regular functions $F$ on $\mathscr U $ with  
$$\|F\|_{  H^p(\mathscr U )}:=\left(\sup_{\varepsilon>0}\int_{{\mathscr H^{n-1} } }|F (\varepsilon, g)|^pd g \right)^{1\over p}<\infty,$$
where $d g $ is the Lebegue measure on $\mathbb{R}^{4n-1}$, which is an invariant measure on $\mathscr H^{n-1}$.
The Cauchy--Szeg\H o projection integral operator $ \mathcal P : L^2(\mathscr H^{n-1} )\rightarrow H^2( \mathscr U )$ is given by
\begin{equation}\label{eq:SzegoP}
    (\mathcal P  f) (t,g)   =
\int_{ \mathscr H^{n-1} }
    K((t,g), g') f (g') dg',  \qquad ( t,g) \in \mathscr U ,
 \end{equation} for $  f\in L^2(\mathscr H^{n-1} )$,
satisfying the following reproducing formula:
\begin{equation}\label{eq:Szego}
    f   = \mathcal P  f^b ,
 \end{equation}whenever $f \in   H^2(\mathscr U )$ and $f^b$ its boundary value on $ \mathscr H^{n-1}  $, where the reproducing
kernel $K((t,g), g')$ is induced from
$S(q,p)$ in \eqref{cauchy-szego} (cf. \eqref{eq:K}).

Just recently, in \cite{CDLWW} we have further obtained an explicit formula for this Cauchy--Szeg\H{o} kernel and then proved that it is a Calder\'on--Zygmund kernel when restricted to $\mathscr H^{n-1}$. Moreover, a suitable version of pointwise lower bound is provided:
``there exist a large positive constant $r_0$ and a positive constant $C$ such that
for every $g\in \mathscr H^{n-1}$, there exists a `twisted truncated sector' $\Omega_g\subset \mathscr H^{n-1}$ such that
$ \inf_{g'\in \Omega_g} {d(g,g')}=r_0 $ and  for every $g_1\in B(g,1)$ and $g_2\in \Omega_g$ we have
$|S(g_1, g_2)|\gtrsim d(g_1,g_2)^{-Q}.$
Moreover, the sector $\Omega_g$ is regular in the sense that $|\Omega_g|=\infty$ and for every
$R_2>R_1>2r_0$
$ \big| \big(B(g,R_2)\backslash B(g,R_1)\big)  \cap \Omega_g\big | \approx  \big| B(g,R_2)\backslash B(g,R_1)\big|$
with the implicit constants  independent of  $R_1,R_2$ and $g$.''

Note that this lower bound yields the characterisations of boundedness and compactness of commutators $[b, \mathcal  P]$ for $b\in$ BMO space and VMO space respectively.

\subsection{Statement of main results}
In this paper we  investigate several fundamental results for the Cauchy--Szeg\H{o} projection, including the
pointwise regularity (higher order) and the non-degenerated property for the Cauchy--Szeg\H{o} kernel. To be more explicit, we obtain the following results. 

\medskip
\noindent Part 1. The pointwise regularity of the Cauchy--Szeg\H{o} kernel.
\begin{thm}\label{main1}
Suppose $g \in \mathscr H^{n-1}\setminus \{0\}$.
$$\left|Y^IK((1,0),g) \right|  \lesssim {1\over (\|g\|+1)^{Q+d(I)}} ,$$
 where $Q=4n+2$ is the homogeneous dimension of $\mathscr H^{n-1}$, $I=(\alpha_1,\cdots, \alpha_{4n-1})\in\mathbb N^{4n-1}$ is a multi-index, $Y^I, d(I)$ and $I$ are defined in \eqref{di}.
\end{thm}

Based on this property, we further have the following result on
Cauchy--Szeg\H{o} projection.

\begin{prop}\label{prop0}  Suppose $0<p<\infty$. For any { $(p,\infty,\alpha)$-atom $a$ on the quaternionic Heisenberg group $\mathscr H^{n-1}$ {\rm (Definition \ref{def atom on Hn})}, }
$\mathcal P (a)$ is an element in $H^p(\mathscr U )$ with
\begin{equation*}
   \|\mathcal P (a)\|_{H^p(\mathscr U )}\leq C_{p,n,\alpha}
\end{equation*}
  for some constant  depending only on $p,n$ and $\alpha$.
\end{prop}
This proposition implies that the   regular Hardy space  $  H^p(\mathscr U )$ is nontrivial, since there are lots of $(p,\infty,\alpha)$-atoms on the quaternionic Heisenberg group $\mathscr H^{n-1}$.

\medskip

\noindent Part 2. Pointwise  lower bound of the Cauchy--Szeg\H{o} kernel.

\smallskip
Letting $t\rightarrow 0$ in \eqref{eq:SzegoP}, we obtain a convolution operator on the quaternionic Heisenberg group, which is also denoted by $\mathcal P$ by abuse of notations,
\begin{align}\label{cs projection gp}
(\mathcal P f)(g)= p.v. \int_{\mathscr H^{n-1}}K(g,h)f(h)dh=p.v. \int_{\mathscr H^{n-1}}K(h^{-1}\cdot g)f(h)dh,
\end{align}
where the kernel $K(g,h):=\lim_{t\rightarrow0} K((t,g), h )$ for $g\neq h $ and $K(g ):=K(g,0) $.
 Note that {\eqref{cs projection gp} holds  whenever $f$ is an $L^2$ function supported in a compact set and
	 $g\not\in {\rm supp}\, f$.}

\begin{thm}\label{main2} 
 $K(g) \not=0,$ for\ all $ g\in \mathscr H^{n-1}\backslash\{0\}.$
\end{thm}\color{black}

Based on this result, we further have the non-degenerated pointwise lower bound.
To begin with, we adapt the work of \cite{Str} on self-similar tilings to obtain a ``nice'' decomposition of $\mathscr H^{n-1}$, denoted by $\tile$, analogous to the decomposition of $\mathbb R^n$ into dyadic cubes in classical harmonic analysis, and describe an analogue of a lemma of Journ\'e \cite{J}. We recall the detail of tiles in Section 3.

The following pointwise lower bound is a refinement of the result in \cite{CDLWW} in the sense that for each given tile $T$, we know exactly the location of the existing tile $\hat T$.

\begin{thm}\label{main3}
Let  $   K = K_1 +K_2 \textbf{i}+K_3 \textbf{j} +K_4 \textbf{k} $.
There exists a positive integer $\mathfrak a_0$ such that:\  

{\rm(1)} for any $T\in \tile_{j}$, there is a unique $T_{\mathfrak a_0}\in \tile_{j+\mathfrak a_0}$ such that $T\subset T_{\mathfrak a_0}$.

{\rm(2)} there exist positive constants $3\leq \mathfrak a_{1}\leq \mathfrak a_{2}$ and $C>0$ such that for any tile $T\in\tile_{j}$, there exists a tile $\hat{T}\in\tile_{j}$ satisfying:
\begin{enumerate}
  \item[(a)] $\hat{T}\subset T_{\mathfrak a_0}$;

  \item[(b)] $\mathfrak a_{1}2^{j}\leq d(\cent{(T)},\cent(\hat{T}))\leq \mathfrak a_{2}2^{j}$;

  \item[(c)] for all $(g,\hat g)\in T\times \hat{T}$, there exists $i\in\{1,2,3,4\}$ such that
                       $|K_i(g,\hat g)|\geq C 2^{-(4n+2)(j)}$ and  $K_i(g,\hat g)$ does not change sign.

\end{enumerate}
\end{thm}

\color{black}

\subsection{Applications of our main results}
\subsubsection{\textbf{Regular Hardy space on the quaternionic Siegel upper half space}}

 The Beltrami--Laplace operator associated to a K\"ahler metric
  is a fundamental tool in the   study of holomorphic $H^p$ functions,
since it annihilates   holomorphic functions \cite[chapter III]{St} \cite{Gel}.
   On    the quaternionic   Siegel upper half space, direct calculation shows that regular functions  do not satisfy  the Beltrami--Laplace equation associated to the quaternionic  hyperbolic metric (some modification may   work).
 But by   identifying the Siegel upper half space with $\mathscr U=\mathbb{R}_+\times\mathscr H^{n-1}$,  we observed that a regular function on $\mathscr U$  satisfies  a more simple equation, the heat equation associated with the {sub-Laplacian} on the  quaternionic Heisenberg group.
\begin{thm}  \label{prop:regular-sublap} An $\mathbb{H}$-valued function $f( t ,g)$  regular on $\mathscr U=\mathbb{R}_+\times\mathscr H^{n-1}$    satisfies
\begin{equation}\label{eq:regular-sublap}
   \left(\partial_{t }+\frac 1{8(n-1)}\triangle_H\right)f=0,
\end{equation}
where $\triangle_H$ is the sub-Laplacian on $ \mathscr H^{n-1}$.
\end{thm}

By using parabolic maximum principle, we show  that an $  H^1(\mathscr U )$ function $f $ is the  {heat kernel integral} of its boundary value as follows:
\begin{equation}\label{eq:heat-rep}
    f ( t,g )=\int_{\mathscr H^{n-1}}h_t(g'^{-1}\cdot g)f^b(   g') d  g',
 \end{equation}where $h_t(g)$ is the heat kernel   $e^{-\frac t{8(n-1)}\triangle_H}$ and $f^b  $ is the radial limit of $f $. Consequently, $f^b  $ belongs to boundary Hardy space $H^p(\mathscr H^{n-1})$.
  This gives us a formula reproducing regular functions directly,  as an approximation to the identity, whose kernel is controlled by  the geometry of the    quaternionic Heisenberg group nicely.

 We point out that our argument can be applied to holomorphic functions
 on Siegel upper half space in $\mathbb{C}^n$ to simplify corresponding part of Geller's proof in \cite{Gel}.

\begin{defn}
A regular function $A$ on $\mathscr U $ is a {\it regular $p$-atom} if
there is a $(p,\infty,\alpha)$-atom  $a$ on $\mathscr H^{n-1}$
  such that $A = \mathcal{P}(a) \in H^p(\mathscr U )$.
\end{defn}

\begin{defn}\label{def Hardy at on Hn}
The atomic Hardy space $H^p_{at}(\mathscr U)$ is the set of all regular  functions of the form
$$\sum_{j=1}^\infty A_j \lambda_j\quad {\rm with}\quad \lambda_j\in \mathbb{H},\quad \sum_{j=1}^\infty |\lambda_j|^p<+\infty,$$
 where each $A_j$ is a regular $p$-atom,  such a series converges to a regular function by Lemma \ref {lem:bound}. Moreover, the norm (quasinorm) of $f\in H^p_{at}(\mathscr U)$ is the infimum of $\Big(\sum_{j=1}^\infty |\lambda_j|^p\Big)^{1\over p}$ taken over all possible decomposition of $f$ in terms of $\sum_{j=1}^\infty A_j \lambda_j$.
\end{defn}
The atomic Hardy space $H^p_{at}(\mathscr U)$ is a right quaternionic vector space (cf. Remark \ref{eq:right}). Next, we have the following characterisation of the regular Hardy space $  H^p(\mathscr U )$.

\begin{thm}\label{thm2} For ${2\over3}< p\leq1$,
$H^p_{at}(\mathscr U ){=} H^p(\mathscr U )$  and they have equivalent quasi-norms.
\end{thm}
We would like to point out that the restriction of $   p$ comes from the subharmonicity of $|f|^p$ for $f\in H^p(\mathscr U )$  which may not   hold for $p< {2\over3}$.
 Note that the class of holomorphic   functions is unique in the sense that $|f|^p$ is subharmonic for any holomorphic   function $f$ and  any $p>0$. Since the Cauchy--Fueter operator in one  quaternionic variable is a kind of  generalized Cauchy--Riemann operator,  $|f|^p$ is not always subharmonic for $p<\frac {2}{3}  $, and so regular  $H^p(\mathscr U )$ space for such a $  p $ will be completely different from regular $H^p(\mathscr U )$ space for  $\frac {2}{3}< p \leq 1$.

For $p<1  $, it does not make sense to say that the boundary value of a regular $H^p$ function  on $\mathscr U$ belongs to the boundary Hardy space $H^p(\mathscr H^{n-1})$, because an element of $H^p(\mathscr H^{n-1})$ may be a distribution. To prove Theorem \ref{thm2}, we show that for a regular $H^p$ function $f$ with ${2\over3}< p\leq1$,  $f(\varepsilon,\cdot)\in H^p_{at}(\mathscr H^{n-1})$ for any $\varepsilon>0$  and their
 $ H^p_{at}(\mathscr H^{n-1})$ quasi-norms are uniformly bounded. Then there exists a
 subsequence $\{f(\varepsilon_k,\cdot) \}$  weakly convergent to some $f^b$ in $ H_{at}^p(\mathscr H^{n-1}) $. On the other hand,
for $\mathfrak f \in L^1(\mathscr H^{n-1})\cap L^\infty (\mathscr H^{n-1})$ and fixed $(t,g)\in\mathscr U$, we can show that 
\begin{equation*}
   S_{(t,g)}(\mathfrak f ):=\mathcal P(\mathfrak f ) (t,g)
\end{equation*} is  well defined and can be extended to a continuous linear functional on $H_{at}^p(\mathscr H^{n-1})$. The regular atomic decomposition will follow from $S_{(t,g)}(f^b) =\lim_{k\to\infty}S_{(t,g)}f(\varepsilon_k,\cdot)$ and substituting  atomic decomposition of $f^b$ on the boundary.

We note that for domains in $\mathbb C^n$, the method of obtaining atomic decomposition for the holomorphic Hardy space via the atomic decomposition on the boundary was first used by Krantz--Li \cite{KL1}, but our techniques here are different from \cite{KL1} in many aspects. 

\subsubsection{\textbf{Singular value estimates for commutators $[b, \mathcal P]$}}

Recall that in \cite{CDLWW} we have obtained the characterisations of boundedness and compactness of commutators $[b, \mathcal  P]$ for $b\in$ BMO space and VMO space respectively. Based on our Theorem \ref{main3}, we provide a deeper study on the singular value estimates commutators $[b, \mathcal  P]$. Recall that in the classical setting, characterisations of singular values and the Schmidt decomposition of commutator of Hilbert transform and Riesz transforms were first studied by Peller in $\mathbb R$ \cite{P} (see also \cite{P2}), then developed by Janson--Wolff in $\mathbb R^n$, $n\geq2$ \cite{JW}, and later on by Rochberg--Semmes \cite{RS0,RS}, via the Schatten classes and the Besov spaces. We summarise the known results  as follows. Let $H$ denote the Hilbert transform and let $R_\ell$ denote the $\ell$-th Riesz transform on $\mathbb R^n$.
\begin{itemize}
\item[(1)] If $n=1$ and $0< p<\infty$, then $[b,H]$ is in Schatten class $S^p$ if and only if the symbol $b$ is in the Besov space $B_{p,p}^{1/p}(\mathbb R)$ \cite{P,P2}.
\item[(2)] Suppose $n\geq2$ and $b\in L^1_{{\rm loc}}(\mathbb R^n)$. When $p>n$, $[b,R_{\ell}]\in S^p$ if and only if $b\in B_{p,p}^{{n}/{p}}(\mathbb R^n)$; when $0<p\leq n$, $[b,R_{\ell}]\in S^p$ if and only if $b$ is a constant \cite{JW,RS}.
\end{itemize}

The Janson--Wolff inequality has close connection to a quantised derivative of Alain Connes introduced in \cite[IV]{Con} in terms of  the (weak) Schatten norm of the commutator \cite{LMSZ}.
A similar Janson--Wolff phenomenon has been demonstrated by Feldman and  Rochberg  \cite{FR} for the Cauchy--Szeg\H o projection on the unit ball and  Heisenberg group via the Hankel operators. Here we establish the analogous result of Feldman and  Rochberg in our setting, based on our main result, Theorem \ref{main3}, and on the recent breakthrough of Fan, Lacey and Li \cite{FLL}, since the theory of Fourier analysis and Hankel operator is not as effective as in the Euclidean setting or the classical Heisenberg group setting.

 \begin{defn}
Let $f\in L_{\rm loc}^{1}(\mathscr H^{n-1})$. Then we say that $f$ belongs to Besov space $B_{p,p}^{4n+2\over p}(\mathscr H^{n-1})$ if
\begin{align*}
{\int_{\mathscr H^{n-1}}\frac{\|f(g \cdot)-f(\cdot)\|_{L^{p}(\mathscr H^{n-1})}^{p}}{\rho(g)^{8n+4}}dg<\infty.}
\end{align*}
\end{defn}

We also recall the definition of the Schatten class $S^{p}$. Note that if $T$ is any compact operator on $L^{2}(\mathscr H^{n-1})$, then $T^{*}T$ is compact, positive and therefore diagonalizable. For $0<p<\infty$, we say that $T\in S^{p}$ if $\{\lambda_{n}\}\in \ell^{p}$, where $\{\lambda_{n}\}$ is the sequence of square roots of eigenvalues of $T^{*}T$ (counted according to multiplicity).

Our main theorem on Schatten class commutators is the following.
\begin{thm}\label{schatten}
Suppose that $0<p<\infty$ and $b\in  L^1_{{\rm loc}}(\mathscr H^{n-1})$. Then

{\rm (1)} for $p>4n+2$, $[b, \mathcal P]\in S^p$ if and only if $b\in B_{p,p}^{\frac{4n+2}{p}}(\mathscr H^{n-1})$;

{\rm (2)} for $0<p\leq 4n+2$, $[b, \mathcal P]\in S^p$ if and only if $b$ is a constant.
\end{thm}

This paper is organised as follows. In Section 2 we provide the necessary preliminaries on
quaternionic Siegel upper half space and the quaternionic Heisenberg group. In Section 3 we give proofs of our main results Theorems \ref{main1}---\ref{main3} and Proposition \ref{prop0}. The application on regular Hardy space on the quaternionic Siegel upper half space will be provided in Section 4 and 5. The application on singular value estimates of the commutator of Cauchy--Szeg\H{o} projection will be given in Section 6.

\section{The flat model of the quaternionic  Siegel upper half space}

The boundary $\partial \mathcal U $ can be identified with the quaternionic Heisenberg group
 $\mathscr H^{n-1}$, which is the  space $\mathbb R^{4n-1}$ equipped with the multiplication given by
\begin{align}\label{law}
({\bm{t}},{\bm{y}})\cdot({\bm{t}'}, {\bm{y}}')=\bigg(\bm t+\bm t'+B(\bm y, \bm y'), {\bm{y}}+{\bm{y}}' \bigg),
\end{align}
\color{black}
where
${\bm{t}}=(t_1, t_2, t_3)$, ${\bm{t}}'=(t'_1, t'_2, t'_3)\in\mathbb R^3$,  ${\bm{y}}=(y_1,y_2,\cdots, y_{4n-4})$, ${\bm{y}}'=(y'_1,y'_2,\cdots, y'_{4n-4})\in\mathbb R^{4n-4}$, $B(\bm y, \bm y')=(B_1(\bm y, \bm y'),B_2(\bm y, \bm y'),B_3(\bm y, \bm y'))$, and
$$B_\alpha(\bm y, \bm y')=2\sum_{l=0}^{n-2}\sum_{j,k=1}^4b_{kj}^\alpha y_{4l+k}y'_{4l+j},\quad \alpha=1,2,3,$$
with
\begin{equation}\label{eq:b}
b^1:=\left( \begin{array}{cccc}
0&1 & 0 & 0\\
-1 & 0 & 0&0\\
0 & 0 & 0&-1\\
0 & 0 &1&0
\end{array}
\right ),
\quad
b^2:=\left( \begin{array}{cccc}
0&0 &1 & 0\\
0 & 0 & 0&1\\
-1 & 0 & 0&0\\
0 &-1 & 0 &0
\end{array}
\right ),
\quad
b^3:=\left( \begin{array}{cccc}
0&0 & 0 &1\\
0 & 0 &-1&0\\
0 &1 & 0&0\\
-1 & 0 & 0 &0
\end{array}
\right ).
\end{equation}
The identity element of $\mathscr H^{n-1}$ is $0\in\mathbb R^{4n-1}$, and the inverse element of $(\bm t, \bm y)$ is $(-\bm t, -\bm y)$.
 For any $g=(\bm t, \bm y)\in \mathscr H^{n-1}$, the homogeneous norm of $g$ is defined by
$$ \|g\|:= (|{\bm{y}}|^4+|\bm t|^2 )^{1\over 4}.$$
 According to \cite{Cy}, for any $g, g'\in \mathscr H^{n-1}$,  $d(g, g'):=\| g'^{-1}\cdot g\| $ is a distance.  We define  balls  on $\mathscr H^{n-1}$ by   $B(g,r):=\{ g'\mid d(g, g')<r \}$.

The quaternionic Heisenberg group $\mathscr H^{n-1}$ is a homogeneous group with dilations
$$\delta_r(\bm t, \bm y)=(r^2\bm t, r \bm y),\quad r>0.$$
Then for any measurable set $E\subset\mathscr H^{n-1}$,
$$|\delta_r(E)|=r^{Q}|E|,$$
 where $Q=4n+2$ is the homogeneous dimension of $\mathscr H^{n-1}$.
Denote by $\tau_h$ the left translation by $h$, i.e. $$\tau_h(g)=h\cdot g.$$
\color{black}

 The following $4n-1$ vector fields are left invariant on $\mathscr H^{n-1}$:
 \begin{align}\label{Y}
 Y_{4l+j} &={\partial\over \partial y_{4l+j}} + 2 \sum_{\alpha=1}^3\sum_{k=1}^4 b_{kj}^\alpha y_{4l+k} {\partial\over\partial t_\alpha},\quad l=0,\cdots, n-2, ~j=1,\cdots, 4,\\
 T_\alpha&={\partial\over\partial t_\alpha},\quad \alpha=1,2,3.\nonumber
 \end{align}
They form a basis for the Lie algebra of left-invariant vector field on $\mathscr H^{n-1}$.
The only nontrivial commutator relations are
\begin{equation}\label{eq:Y-bracket}
   [Y_{4l+k}, Y_{4l'+j}]=4\delta_{ll'}\sum_{\alpha=1}^3b_{kj}^\alpha{\partial\over \partial t_\alpha}, \quad l, ~l'=0,\ldots ,n-2; ~j,~k=1,\ldots,  4.
\end{equation}
 For convenience, we set $Y_{4n-4+\alpha}:=T_\alpha$,  $y_{4n-4+\alpha }:=t_\alpha$, $~\alpha=1,2,3$.  The standard sub-Laplacian on
 $\mathscr H^{n-1}$ is defined by $\triangle_H=\sum_{j=1}^{4n-4}Y_j^2$. For any multi-index   $I=(\alpha_1,\cdots, \alpha_{4n-1})\in\mathbb N^{4n-1}$, we set $Y^I:=Y_1^{\alpha_1}\cdots Y_{4n-1}^{\alpha_{4n-1}}$, and further set
 \begin{align}\label{di}
 d(I):=\sum_{j=1}^{4n-4}\alpha_j+\sum_{k=4n-3}^{4n-1}2\alpha_k ,\qquad |I|:=\sum_{j=1}^{4n-1}\alpha_j, \end{align}
which are called the  homogeneous  degree and topology degree of the differential $Y^I$,
respectively.

\color{black}

We can   identify the Siegel upper half space with $\mathscr U=\mathbb{R}_+\times\mathscr H^{n-1}$ by the following quadratic diffeomorphism:
  \begin{equation}    \label{eq:pi}
 \begin{split}
\pi:  \mathcal U  &\longrightarrow \mathscr U  ,\\
(q_1  , q')&\longmapsto \left (q_1-|q'|^2 , q'\right).
 \end{split} \end{equation}
The diffeomorphism $\pi$ in (\ref{eq:pi}) induces an isomorphism of Hardy spaces
\begin{equation}\label{eq:isomorphism}\begin{split}
   H^p(\mathscr U  )&\longrightarrow   H^p(\mathcal{U} ),\\ f &\mapsto F(q_1  , q' )=f(q_1-|q'|^2 , q'),
\end{split}\end{equation}
with $\|\cdot\|$ preserved.

  \begin{prop}  \label{prop:regular-equiv} An $\mathbb{H}$-valued function $f$ is regular on $\mathcal U $ if and only if $F:=f\circ \pi^{-1}$ satisfies
  \begin{equation}\label{eq:regular}
      \overline{Q}_mF=0,\qquad m=0,\ldots,n-1,
  \end{equation}
on $ \mathscr U$,  where
  \begin{equation*}\begin{split}
  \overline{Q}_0:&=\partial_t +
 \mathbf{i}\partial_{t_{1}}
 + \mathbf{j}\partial_{t_{2}}+ \mathbf{k}\partial_{t_{3}},\\
   \overline{Q}_{l+1 }:&= Y_{4 l +1}+
 \mathbf{i}Y_{4 l +2}
 + \mathbf{j}Y_{4 l  +3}+ \mathbf{k}Y_{4 l +4} ,\quad  l=0,\ldots,n-2.
 \end{split} \end{equation*}
\end{prop}
\begin{proof}  The proof is similar to   \cite[Proposition 5.1]{Shi}.
The inverse of $\pi$ is
\begin{equation}\label{eq:pi-1}\begin{split}
 \pi^{-1}:\mathscr U &\longrightarrow \mathcal U ,\\
     ( t,\bm{t}, {\bm{y}})&\mapsto (x_1,\ldots,x_{4n})=\left(t+ |{\bm{y}}|^2 ,\bm{t},{\bm{y}}\right),
 \end{split}\end{equation}
  where   $\bm{t}=(t_1,t_2,t_3)$. Note that $\overline{\partial}_{q_{l+1 }}+2 q_{l+1 }
\overline{\partial}_{q_{ 1}}$ is a vector field tangential to the hypersurface $\partial \mathcal U $,  since by definition,  we have
$$\left(\overline{\partial}_{q_{l+1}}+2 q_{l+1 }
\overline{\partial}_{q_{ 1}}\right)\rho=0,$$ $l= 1,\cdots,n -1,$
where $\rho $ is the defining function (\ref{eq:boundary}) of $  \partial\mathcal{ U }$. Denote 
{ $\widetilde{q}_{l+1} := y_{4l-3} + y_{4l-2}\mathbf i + y_{4l-1}\mathbf j + y_{4l}\mathbf  k$}, $l = 1, \ldots , n-1$, and $\overline{\partial}_{\widetilde{q}_{l+1}}$ is defined as in (\ref{eq:CF}).
  We  claim that
\begin{equation}\label{eq:pi-Q0}
   \pi^{-1}_*\left(\overline{\partial}_{\widetilde{q}_{l+1}}+2\widetilde{ q}_{l+1 }
\overline{\partial}_{\bm{t}}\right) = \overline{\partial}_{q_{l+1}}+2 q_{l+1 }
\overline{\partial}_{q_{ 1}}
\end{equation}where $\overline{\partial}_{\bm{t}}:=\textbf{i}{\partial}_{t_{1}}+
\textbf{j}{\partial}_{t_{2}}+\textbf{k}
{\partial}_{t_{3}} $, and
\begin{equation}\label{eq:pi-Q1}
   \overline{\partial}_{\widetilde{q}_{l+1}}+2 \widetilde{q}_{l+1 }
\overline{\partial}_{\bm{t}}= \overline{Q}_{l },
\end{equation}$l= 1,\cdots,n -1 $.
      Namely, we have
\begin{equation}\label{eq:pi-Q}
  \pi^{-1}_*   \overline{Q}_{l } = \overline{\partial}_{q_{l+1}}+2 q_{l+1 }
\overline{\partial}_{q_{ 1}} .
\end{equation}

It follows from the definition of $\pi^{-1}$ in
(\ref{eq:pi-1})  that
\begin{equation*}\begin{split}
    \pi^{-1}_*\partial_{t_{\beta}}&=\partial_{x_{ 1+\beta}},\qquad\qquad \qquad \qquad   \beta=1,2,3,
    \\ \pi^{-1}_*\partial_{y_{4(l-1)+j}}&=\partial_{x_{4l+j}}+ 2x_{4l+j} \cdot
\partial_{x_{ 1}},\qquad   j=1,\cdots,4,\quad l=1,\cdots,n-1,
\end{split}\end{equation*}since $x_{4l+j}=y_{4(l-1)+j}$.
Then
\begin{equation*}\begin{aligned}
\pi^{-1}_*\left(\overline{\partial}_{\widetilde{q}_{l+1}}+2\widetilde{q}_{l+1}
\overline{\partial}_{\bm{t}}\right)=&\sum_{j=1}^4\textbf{i}_{j }\left(
{\partial}_{x_{4l+j}}+2x_{4l+j}
{\partial}_{x_{ 1}}\right)+2
 q_{l+1} \left(\textbf{i}{\partial}_{x_{ 2}}+
\textbf{j}{\partial}_{x_{ 3}}+\textbf{k}
{\partial}_{x_{ 4}}\right)\\=&
\overline{\partial}_{q_{l+1}}+2 q_{l+1}
\overline{\partial}_{q_{ 1}}
,
\end{aligned}\end{equation*}
if we denote  $\textbf{i}_1 :=1  $, $\textbf{i}_2: =\textbf{i}  $, $\textbf{i}_3 :=\textbf{j}  $, $\textbf{i}_4: =\textbf{k}  $. (\ref{eq:pi-Q0}) is proved.

To show (\ref{eq:pi-Q1}), consider right multiplication by  $\textbf{i}_\beta $. Noting  that
\begin{align*}
(x_1+x_2\textbf{i}+x_3\textbf{j}+x_4\textbf{k})\textbf{i}
=-x_2+x_1\textbf{i}+x_4\textbf{j}-x_3\textbf{k},\\
(x_1+x_2\textbf{i}+x_3\textbf{j}+x_4\textbf{k})\textbf{j}
=-x_3-x_4\textbf{i}+x_1\textbf{j}+x_2\textbf{k},\\
(x_1+x_2\textbf{i}+x_3\textbf{j}+x_4\textbf{k})\textbf{k}
=-x_4+x_3\textbf{i}-x_2\textbf{j}+x_1\textbf{k},
\end{align*}
we can write
\begin{align}\label{2.8}
(x_1+x_2\textbf{i}+x_3\textbf{j}+x_4\textbf{k})\textbf{i}_\beta= -(b^\beta x)_1-(b^\beta x)_2\textbf{i}-(b^\beta x)_3\textbf{j}-(b^\beta x)_4\textbf{k}=- \sum_{j=1}^4(b^\beta x)_j\textbf{i}_{j },
\end{align}
where $b^\beta$'s are given by (\ref{eq:b}), and $x$ is the column vector $(x_1,x_2,x_3,x_4)^t$.  Then we have \begin{equation*}\begin{aligned}
 \widetilde{q}_{l+1} \overline{\partial}_\mathbf t=& 
 {\left( y_{4(l-1)+1} +\textbf{i} y_{4(l-1)+2} +
\textbf{j} y_{4(l-1)+3} +\textbf{k} y_{4(l-1)+4} \right)}
\left(
\textbf{i}\partial_{t_{1}}+
\textbf{j}\partial_{t_{2}}+\textbf{k}\partial_{t_{3}}\right)\\=&
 \sum_{\beta=1}^3
\sum_{j,k=1}^{4}b_{jk}^\beta y_{4(l-1)+k} \textbf{i}_{j }
\partial_{t_{\beta}},
\end{aligned}\end{equation*}
and so
\begin{equation}\label{eq:Xj}\begin{split}\overline{\partial}_{\widetilde{q}_{l+1}}+ 2\widetilde{q}_{l+1}\overline{\partial}_{\bm{t}}
& =\sum_{j=1}^4\textbf{i}_{j }
{\partial}_{y_{4(l-1)+j}}+\sum_{j=1}^4\textbf{i}_{j }\sum_{\beta=1}^3
\sum_{ k=1}^{4}b_{k j}^\beta y_{4(l-1)+k}
\partial_{t_{\beta}}\\&
=
   Y_{4(l-1)+1}+\mathbf{i} Y_{4(l-1)+2}+\mathbf{j} Y_{4(l-1)+3}+\mathbf{k} Y_{4(l-1)+4},
\end{split}\end{equation}
by the antisymmetry of $b^\beta$'s.
It is also direct to see that
\begin{equation*}
   \pi^{-1}_*( \partial_{t}+
 \mathbf{i}\partial_{t_1}
 + \mathbf{j}\partial_{t_2}+ \mathbf{k}\partial_{t_3})=
\overline{\partial}_{q_{ 1}}.
\end{equation*}

Now the equivalence follows from
\begin{equation*}\begin{split}
 \left.   \overline{Q}_{l }\left(f\circ \pi^{-1}\right)\right|_{\pi(q)}&=\left.\left(\pi^{-1}_* \overline{Q}_{l }\right)f\right|_{q} =\left.\left(\overline{\partial}_{q_{l+1}}+2q_{l+1}
\overline{\partial}_{q_{ 1}} \right )f\right|_{q}=0,
\end{split}\end{equation*}by (\ref{eq:pi-Q}),
and $\left.( \partial_{t}+
 \mathbf{i}\partial_{t_1}
 + \mathbf{j}\partial_{t_2}+ \mathbf{k}\partial_{t_3})(f\circ \pi^{-1})\right|_{\pi(q)}=\left.\overline{\partial}_{q_{ 1}} f\right|_{q}=0$ similarly.
\end{proof}
\begin{rem}\label{eq:right}
   If $f$ is regular, then $f\lambda$ for any $\lambda\in \mathbb{H}$ is also regular. This is because $\overline{Q}_j(f\lambda)=\overline{Q}_j(f)\lambda=0$ by the associativity of quaternions. Thus the Hardy space $H^p(\mathscr U)$ is a right quaternionic vector space.
\end{rem}

By applying the reproducing
formula \eqref{eq:Szego0} to $ F(t+ |{\bm{x}}|^2 ,\bm{t}, {\bm{x}} )  =f( t,\bm{t}, {\bm{x}})$, we get   the reproducing
formula \eqref{eq:Szego} on $\mathscr U=\mathbb{R}_+\times\mathscr H^{n-1}$ with the reproducing kernel given by
\begin{equation}\label{eq:K}
 K((t,g),  g'):=S((t+|{\bm{x}}|^2+\bm{t},{\bm{x}}),(|{\bm{y}}|^2+ \mathbf{s},{\bm{y}})),
\end{equation}
where $g=(\bm{t},{\bm{x}}),  g'=(\mathbf{s},{\bm{y}})\in \mathscr{H}^{n-1}$.

The invariance of the Cauchy--Szeg\H o kernel \cite[Proposition 5.1]{CMW} has the following form.

\begin{prop}\label{prop:Inv-kernel} The Cauchy--Szeg\H o kernel has following invariance properties.
\begin{align*}
&
K((t, \tau_h(g)),  \tau_h(g'))=K((t, g),  g'),\\&
K(  (t, \mathbf{a}g), \mathbf{a} g')=K((t, g),  g'), \\&
K( (r^2 t, \delta_r( g)), \delta_r( g'  )  )r^{Q}=K((t, g),  g'),
\end{align*}
for $g, g'\in \mathscr H^{n-1},$
where $h\in \mathscr H^{n-1}$,
 $\mathbf{a}\in {\rm Sp}(n-1)$    and $t,r>0$.
\end{prop}

\section{Proofs of main results}

\subsection{Proof of Theorem \ref{main1}}

Recall that a function $P$ on a homogeneous Lie group $G$ is called a polynomial if $P\circ \exp$ is a polynomial on its Lie algebra $\mathfrak g$, where $\exp:\mathfrak g \longrightarrow G$ is the exponential mapping. On the  quaternionic Heisenberg group $\mathscr H^{n-1}$, this definition of polynomial coincides with usual one, i.e. polynomials on the  Euclidean space $\mathbb R^{4n-1}$. This is because the  exponential mapping is easily seen to be the identity mapping in this case as follows.
The exponential map $ \exp: \mathbb R^{4n-1}\longrightarrow \mathscr H^{n-1}$ is
\begin{equation*}
 u\mapsto \exp \bigg(\sum_{j=1}^{4n-1} u_j Y_{j } \bigg)(0).
\end{equation*}
$ \exp u$ is given by $\phi^u(1)$, where $\phi^u$ is the integral curve of the vector field $  \sum_{j=1}^{4n-1} u_j Y_{j }  $ starting from the origin. If we write $Y_{j } :=\sum_{m=1}^{4n-1}\mathscr B_{m j}(y){\partial\over\partial y_m}$,
then
\begin{equation*}
\sum_{j=1}^{4n-1} \left. u_j Y_{j }\right|_{\phi^u  (t)} = \sum_{j,m=1}^{4n-1}\mathscr B_{m j}(\phi^u  (t))u_j{\partial\over\partial y_m}
\end{equation*}
and
$\phi^u=( \phi^u_1,\ldots, \phi^u_{4n-1})^t$  is  the solution to
 the Cauchy problem
\begin{equation*}\left\{\begin{split}& \dot{\phi }^u _m(t)=\sum_{j=1}^{4n-1}\mathscr B_{mj}(\phi_u  (t)) u_j,  \qquad m=1,\ldots, 4n-1, \\
&\phi^u  (0)=0
    \end{split}  \right.
\end{equation*}
  (cf. \cite[\S 1.3.4]{BLU} for the exponential map of a homogeneous Lie group).  By the expression of $Y_j$'s in (\ref{Y}), we see that
\begin{equation*}\left\{\begin{split}& \dot{\phi }^u_1(t)=  u_1 ,\\
  &\qquad\vdots  \\
  & \dot{\phi }^u_{4n-4}(t)=  u_{4n-4},  \\
  & \dot{\phi }^u_{4n-4+\alpha}(t)=  u_{4n-4+\alpha}+ 2\sum_{l=0}^{  n-2}  \sum_{j,k=1}^4 b_{kj}^\alpha \phi_{4l+k}^u(t)  u_{4l+j},\qquad \alpha=1,2,3,
 \end{split}\right.\end{equation*}The unique solution is
$
   \phi^u _j (t)= u_jt
$, $j=1,\cdots, 4n-1$, since $(b_{kj}^\alpha   )$  are  skew symmetric. Namely, $ \exp(u)= u$ and the exp is trivial.

Now
every polynomial on $\mathscr H^{n-1}$ can be written uniquely as
$$P=\sum a_I\xi^I,\quad (\xi^I=y^{i_1}_1\cdots y^{i_{4n-4}}_{4n-4}t_1^{i_{4n-3}}t_2^{i_{4n-2}}t_3^{i_{4n-1}}),$$
where all but finitely many of coefficients $a_I$ vanish.
 We call a function $\varphi$ nonisotropic homogeneous of degree $m$ if
$$\varphi\left(\delta_r({\bm{t}}, {\bm{y}}) \right)=r^m\varphi({\bm{t}}, {\bm{y}}),$$
for any $({\bm{t}}, {\bm{y}})\in \mathscr H^{n-1}$ and $r>0$, where $\delta_r({\bm{t}}, {\bm{y}})=(r^2{\bm{t}}, r{\bm{y}})$.  Then
 $\xi^I$ is of homogeneous degree $d(I)=\sum_{j=1}^{4n-4}i_j+2\sum_{k=4n-3}^{4n-1}i_k$. The homogeneous degree of $P$ is defined to be $\max\{d(I): a_I\neq 0\}$.

\begin{proof}[Proof of Theorem \ref {main1}]
  Note that $K((1,0),g)=s(|{\bm{y}}|^2+1-\bm{t})$ for $g=(\bm{t},\bm y)\in \mathscr H^{n-1}\setminus \{0\}$.
We claim that
\begin{equation}\label{ss}
s(\sigma)={P_{2n-1}(x)\over |\sigma|^{4n}},\quad \sigma=x_1+x_2\mathbf i +x_3\mathbf j+x_4 \mathbf k,
\end{equation}
where $P_{2n-1}(x)=P_{2n-1}(x_1,x_2,x_3,x_4)$ is some usual homogeneous polynomial of degree $2n-1$ on $\mathbb{R}^4$.

In fact,
by induction, we can see
\begin{align*}\label{derivative}
{\partial^{l}\over \partial x_1^{l}}\bigg({ 1\over |\sigma|^4} \bigg)
 ={P_l(x)\over |\sigma|^{4+2l}},
\end{align*}
where
 $P_l$ is some quaternionic homogeneous polynomial of degree $l$ on $\mathbb R^4$.
Therefore, by (\ref{s}), we have
\begin{align*}
s(\sigma)
&= c_{n-1}{\partial^{2(n-1)}\over \partial x_1^{2(n-1)}}\bigg({ 1\over |\sigma|^4} \bigg)  {\overline \sigma} +
{(2n-2)}c_{n-1}{\partial^{2n-3}\over \partial x_1^{2n-3}}\bigg({ 1\over |\sigma|^4}\bigg) \\
&={P_{2n-2}(x)\over |\sigma|^{4(n-1)+4}}\overline \sigma+{P_{2n-3}(x)\over |\sigma|^{2(2n-3)+4}}\overline \sigma\\
&={P_{2n-1}(x)\over |\sigma|^{4n}}.
\end{align*}

Next, we claim that for any $g=(\bm{t}, {\bm{y}})\in \mathscr H^{n-1}$, $I=(\alpha_1,\cdots, \alpha_{4n-1})\in\mathbb N^{4n-1}$, we have
\begin{equation}\label{YH}
Y^Is\big(|{\bm{y}}|^2+1+\bm{t}\big)={H_{4n-2+4|I|-d(I)}\left({\bm{t}}, {\bm{y}}\right)\over\left [\left(|{\bm{y}}|^2+1\right)^2+|{\bm{t}}|^2\right]^{2n+|I|}},
\end{equation}
for some  nonisotropic polynomial $H_{4n-2+4|I|-d(I)}$ with homogeneous degree  $4n-2+4|I|-d(I)$.

In fact, by \cite[(2.8)]{CDLWW} and \eqref{ss}, we can see, for $g=(\bm{t}, {\bm{y}})\in\mathscr H^{n-1}$,
\begin{align*}
s\big(|{\bm{y}}|^2+1+\bm{t}\big)&={P_{2n-1}\left(|{\bm{y}}|^2+1, t_1, t_2, t_3\right)\over \left[\left(|{\bm{y}}|^2+1\right)^2+|{\bm{t}}|^2\right]^{2n}}
={\sum\limits_{|\gamma|=2n-1}C_\gamma(|{\bm{y}}|^2+1)^{\gamma_1} t^{\gamma_2}_1 t^{\gamma_3}_2t^{\gamma_4}_3\over \left[\left(|{\bm{y}}|^2+1\right)^2+|{\bm{t}}|^2\right]^{2n}}\\
&={\sum\limits_{|\gamma|=2n-1}\sum\limits_{k=0}^{\gamma_1}A_\gamma\tbinom{\gamma_1}{k}|{\bm{y}}|^{2k}t^{\gamma_2}_1 t^{\gamma_3}_2t^{\gamma_4}_3
\over \left[\left(|{\bm{y}}|^2+1\right)^2+|{\bm{t}}|^2\right]^{2n}}\\
&=:  {H_{4n-2}({\bm{t}}, {\bm{y}})\over \left[\left(|{\bm{y}}|^2+1\right)^2+|{\bm{t}}|^2\right]^{2n} } ,
\end{align*}
for some   nonisotropic polynomial $H_{4n-2}$ with homogeneous degree $4n-2$. Then
\begin{align*}
Y_{4l+j} s\big(|{\bm{y}}|^2+1+{\bm{t}}\big)
&=\frac{{\partial\over \partial y_{4l+j}} H_{4n-2}({\bm{t}}, {\bm{y}})+2\sum\limits_{\alpha=1}^3\sum\limits_{k=1}^4b_{kj}^{\alpha}y_{4l+k}{{\partial\over \partial t_{\alpha}}}H_{4n-2}({\bm{t}}, {\bm{y}})}{\left[\left(|{\bm{y}}|^2+1\right)^2+|{\bm{t}}|^2\right]^{2n}}\\
&\qquad -\frac{2nH_{4n-2}({\bm{t}}, {\bm{y}})}{\left[\left(|{\bm{y}}|^2+1\right)^2+|{\bm{t}}|^2\right]^{2n+1}}\left( 4|{\bm{y}}|^2y_{4l+j}+4\sum_{\alpha=1}^3\sum_{k=1}^4b_{kj}^{\alpha}y_{4l+k}t_{\alpha}\right)\\
& =:\frac{H_{4n+1}({\bm{t}}, {\bm{y}})}{[(|{\bm{y}}|^2+1)^2+|{\bm{t}}|^2]^{2n+1}}
\end{align*}
for some  nonisotropic polynomial $H_{4n+1}$ with homogeneous degree  $ 4n+1 $.
And
\begin{align*}
T_{\alpha}s\big(|{\bm{y}}|^2+1+{\bm{t}}\big)
&={\partial\over\partial t_\alpha}\left\{{H_{4n-2}({\bm{t}}, {\bm{y}})\over \left[\left(|{\bm{y}}|^2+1\right)^2+|{\bm{t}}|^2\right]^{2n}}\right\}\\
&=\frac{{\partial\over \partial t_\alpha} H_{4n-2}({\bm{t}}, {\bm{y}})}{\left[\left(|{\bm{y}}|^2+1\right)^2+|{\bm{t}}|^2\right]^{2n}}-\frac{4nt_\alpha H_{4n-2}({\bm{t}}, {\bm{y}})}{\left[\left(|{\bm{y}}|^2+1\right)^2+|{\bm{t}}|^2\right]^{2n+1}}\\
&=:\frac{H_{4n}({\bm{t}}, {\bm{y}})}{\left[\left(|{\bm{y}}|^2+1\right)^2+|{\bm{t}}|^2\right]^{2n+1}}
\end{align*}
for some   nonisotropic polynomial $H_{4n}$ with homogeneous degree  $ 4n$.

Now assume that
$$Y^Is(|{\bm{y}}|^2+1+{\bm{t}})=\frac{H_{4n-2+4|I|-d(I)}({\bm{t}}, {\bm{y}})}{\left[\left(|{\bm{y}}|^2+1\right)^2+|{\bm{t}}|^2\right]^{2n+|I|}}$$
for some nonisotropic polynomial $H_{4n-2+4|I|-d(I)}$ of degree $4n-2+4|I|-d(I)$. Then
\begin{align*}
 Y_{4l+j}Y^Is(|{\bm{y}}|^2+1+{\bm{t}})
&=\frac{{\partial\over \partial y_{4l+j}} H_{4n-2+4|I|-d(I)}({\bm{t}}, {\bm{y}})+2\sum\limits_{\alpha=1}^3\sum\limits_{k=1}^4b_{kj}^{\alpha}y_{4l+k}{{\partial\over \partial t_{\alpha}}}H_{4n-2+4|I|-d(I)}({\bm{t}}, {\bm{y}})}{\left[\left(|{\bm{y}}|^2+1\right)^2+|{\bm{t}}|^2\right]^{2n+|I|}}\\
&\quad -\frac{2nH_{4n-2+4|I|-d(I)}({\bm{t}}, {\bm{y}})}{[(|{\bm{y}}|^2+1)^2+|{\bm{t}}|^2]^{2n+|I|+1}}\Big( 4|{\bm{y}}|^2y_{4l+j}+4\sum_{\alpha=1}^3\sum_{k=1}^4b_{kj}^{\alpha}y_{4l+k}t_{\alpha}\Big)\\
&=\frac{H_{4n+1+4|I|-d(I)}({\bm{t}}, {\bm{y}})}{\left[\left(|{\bm{y}}|^2+1\right)^2+|{\bm{t}}|^2\right]^{2n+|I|+1}},
\end{align*}
where $l=0,\cdots, n-2$, $j=1,\cdots, 4.$
And
\begin{align*}
T_{\alpha}Y^Is(|{\bm{y}}|^2+1+{\bm{t}})
&={\partial\over\partial t_\alpha}\left\{{H_{4n-2+4|I|-d(I)}({\bm{t}}, {\bm{y}})\over \left[\left(|{\bm{y}}|^2+1\right)^2+|{\bm{t}}|^2\right]^{2n+|I|}}\right\}\\
&=\frac{{\partial\over \partial t_\alpha} H_{4n-2+4|I|-d(I)}({\bm{t}}, {\bm{y}})}{\left[\left(|{\bm{y}}|^2+1\right)^2+|{\bm{t}}|^2\right]^{2n+|I|}}-\frac{2(2n+|I|)t_\alpha H_{4n-2+4|I|-d(I)}({\bm{t}}, {\bm{y}})}{\left[\left(|{\bm{y}}|^2+1\right)^2+|{\bm{t}}|^2\right]^{2n+|I|+1}}\\
&=:\frac{H_{4n+4|I|-d(I)}({\bm{t}}, {\bm{y}})}{\left[\left(|{\bm{y}}|^2+1\right)^2+|{\bm{t}}|^2\right]^{2n+|I|+1}}.
\end{align*}
The claim is proved by induction.

By \eqref{YH} we can see
\begin{align*}
\left|Y^Is(|{\bm{y}}|^2+1+{\bm{t}})\right|&={\left|H_{4n-2+4|I|-d(I)}({\bm{t}}, {\bm{y}})\right|\over \left[\left(|{\bm{y}}|^2+1\right)^2+|{\bm{t}}|^2\right]^{2n+|I|}}
\leq\frac{C\sum\limits_{m=0}^{4n-2+4|I|-d(I)}\|({\bm{t}}, {\bm{y}})\|^{ m}}{\left[\left(|{\bm{y}}|^2+1\right)^2+|{\bm{t}}|^2\right]^{2n+|I|}}\\
&\leq C\frac{1}{\left[\left(|{\bm{y}}|^2+1\right)^2+|{\bm{t}}|^2\right]^{n+{1\over 2}+{1\over 4}d(I)}}\\
&\leq  C{1\over (\|g\|+1)^{4n+2+d(I)}}.
\end{align*}

The proof of Theorem  \ref{main1} is complete.
\end{proof}

\subsection{Proof of Proposition \ref{prop0}}

To begin with, we first recall the atoms and atomic Hardy space on $\mathscr H^{n-1}$, which is a special case in \cite{FoSt}.
\begin{defn}\label{def atom on Hn}
 We say a quaternion  valued function  $f$ on $\mathscr H^{n-1}$ is a {\it $(p,\infty,\alpha)$-atom}  if\\
{\rm (i)}  there exists a ball $B(g_0, r)$ such that $f\in L^\infty(B(g_0, r))$ and  $\operatorname{supp} f\subset B(g_0, r)$;\\
{\rm (ii)} $\|f\|_{L^\infty(\mathscr H^{n-1})} \leq \left| B(g_0, r)\right|^{ -{1\over p}}$;\\
{\rm (iii)} $\int_{\mathscr H^{n-1}}P(g)f(g)dg=0$,\\
for any quaternion  polynomial $P$ of homogeneous degree $\leq\alpha$, where $\alpha\geq [Q({1\over p}-1)]$, $Q=4n+2$ is the homogeneous dimension of $\mathscr H^{n-1}$ (it is equivalent to the vanishing for any scalar polynomials $P$ of homogeneous degree $\leq\alpha$).
\end{defn}

\begin{defn}\label{def Hardy at on Hn}
The atomic Hardy space $H^p_{at}(\mathscr H^{n-1})$ is the set of all quaternion  valued  tempered distributions of the form  $$\sum_{j=1}^\infty  f_j\lambda_j\quad {\rm with}\quad \lambda_j\in \mathbb{H},\quad \sum_{j=1}^\infty |\lambda_j|^p<+\infty,$$
$($the sum converges in the topology of $\mathcal{S}'$$)$, where each $f_j$ is a $(p,\infty,\alpha)$-atom.
Moreover, the norm (quasinorm) of $f\in H^p_{at}(\mathscr H^{n-1})$ is the infimum of $\Big(\sum_{j=1}^\infty |\lambda_j|^p\Big)^{1\over p}$ taken over all possible decomposition of $f$ in the form $f=\sum_{j=1}^\infty f_j \lambda_j$.
\end{defn}
 The atomic Hardy space $H^p_{at}(\mathscr H^{n-1})$ is a right quaternionic vector space.

\color{black}
\begin{prop}\label{prop2}
For each positive integer $k$, there exist $C_k>0$ and $b>0$ such that for all quaternion valued function  $f \in C^{k+1}(\mathscr H^{n-1})$ and all $g, g'\in \mathscr H^{n-1}$,
$$|f(g\cdot g')-P( g')|\leq C_k\| g'\|^{k+1} \sup_{\substack{ \|h\|\leq b^{k+1}\| g'\|,\\ d(I)=k+1}}
\left| Y^{I}f (g\cdot h)\right|,$$
where $P$ is the Taylor polynomial of $f$ at $g$ of homogeneous degree $k$.
 \end{prop}

\begin{proof}
Recall that for a real function  $f$ on   $\mathscr H^{n-1}$,
the  left Taylor polynomial of $f$ at $g$ with  homogeneous degree $k$ is the unique  polynomial $P$ of $f$   with homogeneous degree $k$ such that
$Y^IP(0)=Y^If(g)$ for all $d(I)\leq k$.
Write
$$f(g\cdot g')=f_1(g\cdot g')+f_2(g\cdot g')\mathbf i +f_3(g\cdot g')\mathbf j+f_4(g\cdot g') \mathbf k.$$
Then, by \cite[Corollary 1.44]{FoSt}, there exist $C_k>0$ and $b>0$ such that for each $f_l$, $l=1,2,3,4$,
$$\left|f_l(g\cdot g')-P^{l}( g') \right|\leq C_k\| g'\|^{k+1}\sup_{\substack{\|h\|\leq b^{k+1}\| g'\|\\d(I)=k+1}}
\left| Y^{I}f_{l}(g\cdot h)\right|,\quad l=1,2,3,4,$$
where $P^{l}$ is the left Taylor polynomial of $f_l$ at $g$ with homogeneous degree $k$.
Let
$$P( g')=P^1( g')+P^2( g')\mathbf i +P^3( g')\mathbf j+P^4( g')\mathbf k.$$
Then we have
\begin{equation}
\begin{split}
\left|f(g\cdot g')-P( g')\right|&=\left( \sum_{l=1}^4\left|f_l(g\cdot g')-P^{l}( g') \right|^2\right)^{1\over 2}
\leq  C_k\sum_{l=1}^4\| g'\|^{k+1}\sup_{\substack{\|h\|\leq b^{k+1}\| g'\|,\\ d(I)=k+1}}
 \left| Y^{I}f_{l}(g\cdot h)\right|  \\
 &\leq 4C_k\| g'\|^{k+1}\sup_{\substack{\|h\|\leq b^{k+1}\| g'\|,\\ d(I)=k+1}}
 \left| Y^{I}f (g\cdot h)\right|.
\end{split}
\end{equation}

The proof of Proposition \ref{prop2} is complete.
\end{proof}

\begin{lem}\label{lem1}
$\mathcal P(a)(1,\cdot)$ is uniformly bounded on $L^p(\mathscr H^{n-1})$, $0<p\leq 1$, for any $(p,\infty,\alpha)$-atom $a$.
\end{lem}

\begin{proof}
Let  $a$ be any $(p,\infty,\alpha)$-atom with support in some ball $B=B(g_0,r)\subset \mathscr H^{n-1} $, $r>0$, and $\lambda:=\max\{2b^{\alpha+1}, 2\}$, where $b$ is the constant in Proposition \ref {prop2}.

Write
\begin{equation}\begin{split}
\int_{\mathscr H^{n-1}}\left| \mathcal P(a)(1, g)\right|^pdg
&=\int_{ B(g_0, \lambda r)}\left| \mathcal P(a)(1, g)\right|^pdg
+\int_{\mathscr H^{n-1}\setminus B(g_0, \lambda r)}\left| \mathcal P(a)(1, g)\right|^pdg
:=I_1+I_2.
\end{split}\end{equation}

For the term $I_1$, by H\" older's inequality, we have
\begin{align*}
I_1&=\int_{B(g_0, \lambda r)}\left| \mathcal P(a)(1, g)\right|^pdg
\leq\left|B(g_0, \lambda r) \right|^{1-{p\over 2}}\Big(\int_{B(g_0, \lambda r)}\left| \mathcal P(a)(1, g)\right|^2dg \Big)^{p\over 2}\\
&\leq \left|B(g_0, \lambda r) \right|^{1-{p\over 2}}\Big(\int_{\mathscr H^{n-1}}\left|a(g)\right|^2dg \Big)^{p\over 2}\\
&\leq C\left|B(g_0, \lambda r) \right|^{1-{p\over 2}}\Big( \left|B(g_0, r) \right|^{-{2\over p}+1}\Big)^{p\over 2}\\
& \lesssim 1,
\end{align*}
where the second inequality follows from the fact that $\mathcal P$ is an orthogonal projection operator from $L^2(\mathscr H^{n-1})$ to $H^2(\mathscr U)$.

For the term $I_2$,
by Proposition \ref{prop2} and the vanishing moment condition of $(p,\infty,\alpha)$-atom, we have
\begin{align*}
I_2&=\int_{\mathscr H^{n-1}\setminus B(g_0, \lambda r)}\left|\mathcal P(a)(1, g) \right|^pdg\\
&=\int_{\mathscr H^{n-1}\setminus B(g_0, \lambda r)}\left|
\int_{ {\mathscr H^{n-1}}}K((1, g),  g')a( g')d g'
 \right|^pdg\nonumber\\
 &=\int_{\mathscr H^{n-1}\setminus B(g_0, \lambda r)}\left|
\int_{B(g_0, r)}K\left((1,g),  g'\right)a( g')d g'
 \right|^pdg\nonumber\\
 &=\int_{\mathscr H^{n-1}\setminus B(g_0, \lambda r)}\left|
\int_{B(0, r)}K\left((1,g), g_0\cdot g''\right)a(g_0\cdot g'')d g''
 \right|^pdg\nonumber\\
  &=\int_{\mathscr H^{n-1}\setminus B(g_0, \lambda r)}\left|
\int_{B(0, r)}\left[K\left((1,g), g_0\cdot g''\right)-P( g'')\right]a(g_0\cdot g'')d g''
+ \int_{B(0, r)}P( g'')a(g_0\cdot g'')d g''
\right|^pdg\nonumber\\
&=\int_{\mathscr H^{n-1}\setminus B(g_0, \lambda r)}
\left|
\int_{B(0, r)}\left[K\left((1,g), g_0\cdot g''\right)-P( g'')\right]a(g_0\cdot g'')d g''
\right|^pdg\nonumber\\
&\leq \|a\|^p_{L^\infty(B(g_0,r))}\int_{\mathscr H^{n-1}\setminus B(g_0, \lambda r)}\left(
\int_{B(0, r)}\left|K\left((1,g), g_0\cdot g''\right)-P( g'')\right|d g''
\right)^pdg\nonumber\\
&\lesssim \|a\|^p_{L^\infty(B(g_0,r))}\int_{\mathscr H^{n-1}\setminus B(g_0, \lambda r)}\left(
\int_{B(0, r)}\| g''\|^{ {\alpha+1}} \sup_{\substack{ \|h\|\leq b^{\alpha+1}\| g''\|,\nonumber\\
d(I)=\alpha+1}}\left|Y^IK \left((1,g), g_0\cdot h\right)\right|d g''
\right)^pdg\nonumber\\
&\lesssim |B(g_0,r)|^{-1}r^{(\alpha+1)p}
\int_{\mathscr H^{n-1}\setminus B(g_0, \lambda r)}\left(
\int_{B(0, r)} \sup_{\substack{ \|h\|\leq b^{\alpha+1}\| g''\|,\nonumber\\
d(I)=\alpha+1}}\left|Y^IK \left((1,g), g_0\cdot h\right)\right|d g''
\right)^pdg,\nonumber
\end{align*}
where $P$ is the Taylor polynomial of $K((1,g),\cdot)$ at $g_0$ with homogeneous degree $\alpha$.

Since
$$K\left((1,g),  g'\right)=K\left((1,0), g^{-1}\cdot g'\right)=s(|{\bm{y}'}|^2+1-\bm{t}') $$
with $(\bm{t}', {\bm{y}'})=g^{-1}\cdot g'$, by left invariance, we have
$$Y^I_{( g')}\left[K\left((1,g),  g'\right) \right]=Y^I_{( g')}\left[K\left((1,0), g^{-1}\cdot g'\right)\right]
=Y^I_{(\bm{t}',{\bm{y}'})}\left[s\left(|{\bm{y}'}|^2+1-\bm{t}'\right) \right]\big |_{(\bm{t}', {\bm{y}'})=g^{-1}\cdot g'}.$$
Then, by Theorem \ref{main1}, we have
$$\left|Y^I_{( g')}\left[K\left((1,g),  g'\right) \right]\right|
\leq C d(0, g^{-1}\cdot g')^{-Q-d(I)}.$$
For any $g\in\mathscr H^{n-1}\setminus B(g_0, \lambda r)$ and $ g'\in B(g_0, b^{\alpha+1}r)$,
$$ d(0, g^{-1}\cdot g')= d(g,  g')\geq d(g, g_0)- d( g', g_0)\geq {1\over 2} d(g, g_0).$$

Consequently,
\begin{align*}
I_2&\lesssim |B(g_0,r)|^{-1}r^{(\alpha+1)p}
\int_{\mathscr H^{n-1}\setminus B(g_0, \lambda r)}\left(
\int_{B(0, r)} \sup_{\|h\|\leq b^{\alpha+1}\| g''\|}
 d(0,g^{-1}(g_0\cdot h))^{-4n-3-\alpha}d g''
\right)^pdg\\
&\lesssim |B(g_0,r)|^{-1}r^{(\alpha+1)p}|B(0,r)|^{p}
\int_{\mathscr H^{n-1}\setminus B(g_0, \lambda r)} d(g, g_0)^{(-4n-3-\alpha)p}
dg\\
&\lesssim r^{Q(p-1)+(\alpha+1)p}\int_{r}^{\infty}\int_S \tau^{(-4n-3-\alpha)p}\tau^{Q-1}d\sigma d\tau\\
&\lesssim r^{Q(p-1)+(\alpha+1)p+(-4n-3-\alpha)p+Q}\\
&\lesssim 1,
\end{align*}
where  $Q=4n+2$. This completes the proof.
\end{proof}

\begin{proof}[Proof of Proposition \ref{prop0}]
Let  $a$ be any $(p,\infty,\alpha)$-atom with support in some ball $B=B(g_0,r)\subset \mathscr H^{n-1} $ and $r>0$. By the translation invariance of Cauchy--Szeg\H o kernel in Proposition \ref{prop:Inv-kernel}, we have
\begin{align*}
\int_{\mathscr H^{n-1}}|\mathcal P(a)(\varepsilon,   h)|^pd  h
&=\int_{\mathscr H^{n-1}}|\mathcal P(a)(\varepsilon, g_0\cdot g)|^pdg\\
&=\int_{\mathscr H^{n-1}}\Big|\int_{\mathscr H^{n-1}}K((\varepsilon, g_0\cdot g),  g')a( g')d g'\Big|^pdg\\
&=\int_{\mathscr H^{n-1}}\Big|\int_{\mathscr H^{n-1}}K((\varepsilon, g_0\cdot g), g_0\cdot\tilde g')a(g_0\cdot\tilde g')d\tilde g'\Big|^pdg\\
&
=\int_{\mathscr H^{n-1}}\Big|\int_{\mathscr H^{n-1}}K((\varepsilon,g), \tilde g')a(g_0\cdot\tilde g')d\tilde g'\Big|^pdg
 \\
&=\int_{\mathscr H^{n-1}}\Big|\int_{\mathscr H^{n-1}}K((\varepsilon, g), \tilde g')\tilde a(\tilde g')d\tilde g'\Big|^pdg,
\end{align*}
where $\tilde a(\tilde g'):=a(g_0\cdot\tilde g')$ is supported on $B(0,r)$, and it is clear that
$$\|\tilde a\|_\infty=\|a\|_\infty\leq |B(g_0,r)|^{-{1\over p}}=|B(0,r)|^{-{1\over p}}.$$
 For any quaternion  polynomial $P$
of homogeneous degree $\alpha$,
since $P(g_0\cdot)$ is also a polynomial of the same order, we have
\begin{equation}\begin{split}
\int_{\mathscr H^{n-1}}P(g)\tilde a(g)dg&=\int_{\mathscr H^{n-1}}P(g)a(g_0\cdot g)dg=\int_{\mathscr H^{n-1}}P(g_0^{-1}\cdot g') a( g')d g'=0.
\end{split}\end{equation}
Therefore, $\tilde a(\tilde g')$ is a  $(p,\infty,\alpha)$-atom  centered at $0$. So it's sufficient to prove the result for $(p,\infty,\alpha)$-atom supported on $B(0,r)$.  Now by dilation invariance of Cauchy--Szeg\H o kernel in Proposition \ref{prop:Inv-kernel}, we have
\begin{align*}
\int_{\mathscr H^{n-1}}|\mathcal P(a)(\varepsilon,   h)|^pd  h
&=\int_{\mathscr H^{n-1}}\Big|\int_{\mathscr H^{n-1}}K((\varepsilon, \delta_{\sqrt{\varepsilon}} \tilde g), \delta_{\sqrt{\varepsilon}} g'')\tilde a(\delta_{\sqrt{\varepsilon}} g'')\varepsilon^{Q\over 2} d g''\Big|^p\varepsilon^{Q\over 2}d\tilde g\\
&= \int_{\mathscr H^{n-1}}\Big|\int_{\mathscr H^{n-1}}K((1, \tilde g),  g'')\tilde a(\delta_{\sqrt{\varepsilon}} g'')\varepsilon^{Q\over 2p} d g''\Big|^pd\tilde g \\
&=\int_{\mathscr H^{n-1}}\Big|\int_{\mathscr H^{n-1}}K((1, \tilde g),  g'')\tilde {\tilde a}( g'') d g''\Big|^pd\tilde g,
\end{align*}
where $\tilde {\tilde a}( g'') =\varepsilon^{Q\over 2p}\tilde a(\delta_{\sqrt{\varepsilon}} g'')$ is supported on $B(0,{r\over{\sqrt{\varepsilon}}})$ and
$$\big| \tilde {\tilde a}( g'') \big|\leq \varepsilon^{Q\over2 p} |B(0,r)|^{-{1\over p}}
=\big|B\big(0,{r\over{\sqrt{\varepsilon}}}\big)\big|^{-{1\over p}}. $$

For any quaternion  polynomial $P$
of homogeneous degree $\alpha$, 
since $P(\delta_r\cdot)$ is also a polynomial of the same order, we have
\begin{equation}\begin{split}
\int_{\mathscr H^{n-1}}P( g)\tilde {\tilde a}( g)d g
=\int_{\mathscr H^{n-1}}P( g)\varepsilon^{Q\over 2p}\tilde a(\delta_{\sqrt{\varepsilon}} g)d g
=\varepsilon^{{Q\over 2p}-{Q\over 2}}\int_{\mathscr H^{n-1}}P(\delta_{{1\over\sqrt{\varepsilon}}} g')\tilde a( g')d g'=0.
\end{split}\end{equation}
 That is to say $\tilde {\tilde a}$ is a $(p,\infty,\alpha)$-atom  supported on $B(0,{r\over{\sqrt{\varepsilon}}})$, and
\begin{equation}
\int_{\mathscr H^{n-1}}|\mathcal P(a)(\varepsilon,   h)|^pd  h
=\int_{\mathscr H^{n-1}}|\mathcal P(\tilde {\tilde a})(1, \tilde g)|^pd\tilde g.
\end{equation}
Therefore, it's sufficient to prove the result for $\varepsilon=1$ and $(p,\infty,\alpha)$-atom centered at $0$.
Then, by Lemma \ref{lem1}, we can see that
$$\int_{\mathscr H^{n-1}}\left|\mathcal P(a)(\varepsilon,  h) \right|^pd  h\leq C_{p,n,\alpha}.$$
Therefore,
$$\|\mathcal P(a)\|_{H^p(\mathscr U)}=\sup_{\varepsilon>0}\int_{\mathscr H^{n-1}}\left|\mathcal P(a)(\varepsilon,  h) \right|^pd  h\leq C_{p,n,\alpha}.$$

\color{black}
The proof of Proposition \ref{prop0} is complete.
\end{proof}

\subsection{Proof of Theorem \ref{main2}}

Note that  by direct calculation, we can see that
$$\bar{\sigma}\left(x_1+x_2{\bf i}\right)\sigma=x_1+x_2\left[\left(2y_2^2-1\right){\bf i}+2y_2y_3{\bf j}+2y_2y_4 {\bf k}\right],$$
if $\sigma=y_2{\bf i}+y_3{\bf j}+y_4{\bf k}$ with $|\sigma|=1$ (cf. \cite[(5.13)]{CMW}).
Therefore, for given $\xi=\xi_1+\xi_2{\bf i}+\xi_3{\bf j}+\xi_4\mathbf{k}\in\mathbb{H}_+$ with $\xi_3 \xi_4\neq 0$ (otherwise, we already have $\xi=\xi_1+\xi_2{\bf i}$), if we choose $x_1=\xi_1,~x_2=-|\operatorname{Im}\xi|$, where $|\operatorname{Im}\xi|:=(\xi_2^2+\xi_3^2+\xi_4^2)^{1\over 2}$, and
\begin{equation*}\left\{\begin{split}
y_2&=\bigg[ {1\over 2}\Big(1-{\xi_2\over |\operatorname{Im \xi}|} \Big)\bigg]^{{1\over 2}},\\
y_3&=-{\xi_3\over 2\left|{\operatorname{Im \xi}}\right|}\bigg[ {1\over 2}\Big(1-{\xi_2\over |\operatorname{Im \xi}|} \Big)\bigg]^{-{1\over 2}},
\\
y_4&=-{\xi_4\over 2\left|{\operatorname{Im \xi}}\right|}\bigg[ {1\over 2}\Big(1-{\xi_2\over |\operatorname{Im \xi}|} \Big)\bigg]^{-{1\over 2}},
    \end{split}  \right.
\end{equation*}which satisfy $y_2\neq0$ and $y_2^2+y_3^2+y_4^2=1$,
  then we have
$\bar{\sigma}\left(x_1+x_2{\bf i}\right)\sigma=\xi$, i.e.
$$\sigma\xi\bar\sigma=x_1+x_2{\bf i}=\xi_1-|\operatorname{Im}\xi|{\bf i}.$$
Hence, by $s(\sigma\xi\bar\sigma)=\sigma s(\xi)\bar\sigma$  for any $\sigma\in\mathbb H$ with $|\sigma|=1$  \cite[(5.11)]{CMW}, we have
\begin{equation*}\label{eq:s-xi}
   s(\xi)=\bar\sigma s(\xi_1-|\operatorname{Im}\xi|{\bf i})\sigma.
\end{equation*}
Thus  for any $g=(\bm t, \bm y)\in\mathscr H^{n-1}$, 
we have
\begin{equation}\label{eq:s-K}K(g)=s(|\bm y|^2+\bm t)=\bar\sigma s(|\bm y|^2-|\bm t|{\bf i})\sigma.\end{equation}

On the other hand, by \cite[(3.6)]{CDLWW}, for any $\alpha=x_1+x_2 {\bf i}$ with $x_1>0 $, we have
\begin{align}\label{ssigma1}
s(\alpha)
&= c_{n-1}\sum_{k=0}^{2n-2}   (2n-2)!  (2n-k-1) (k+1)    {\bar\alpha\over z^{2n-k}   \bar z^{k+2}} \\
&\quad -  c_{n-1}\sum_{k=0}^{2n-3}   (2n-2)!  (2n-k-2) (k+1)   {1\over z^{2n-k-1}  \bar z^{k+2}},\nonumber
\end{align}
where $ z =x_1+|x_2|{\bf i} $. Now  when $x_2<0 $, write $\alpha=re^{i\theta}$ with $\theta\in ({3\pi\over 2},{2\pi} )$, $r>0$. Then,
\begin{align}\label{ssigma1}
s(\alpha)=s(re^{i\theta})
&= \frac {c_{n-1}(2n-2)!}{r^{2n+1}}  \sum_{k=0}^{2n-2}   (2n-k-1) (k+1) e^{ i(2n-2k-3 )\theta}  \nonumber \\
&\quad -   \frac {c_{n-1}(2n-2)!}{r^{2n+1}} \sum_{k=0}^{2n-2}   (2n-k-2) (k+1)  e^{ i(2n-2k-3)\theta}\\
 &= \frac {c_{n-1}(2n-2)!e^{ i(2n -3)\theta}}{r^{2n+1}}  \sum_{k=0}^{2n-2}     (k+1) e^{-2 k i \theta}\nonumber   \\
  &=  c_{n-1}(2n-2)!e^{ i(2n -3 )\theta}  \frac {(2n-1)t^{2n}-2n t^{2n-1}+1}{ {r^{2n+1}}  (1-t)^2}\nonumber
\end{align}
by using the identity
\begin{equation*}
   \sum_{k=0}^{2n-2}     (k+1)t^k=\frac {(2n-1)t^{2n}-2n t^{2n-1}+1}{(1-t)^2}
\end{equation*}with $t=e^{-2   i \theta} $. Note that for $t^{2n}\neq1$,
\begin{equation*}
   | (2n-1)t^{2n}-2n t^{2n-1}+1|\geq  2n -| (2n-1)t^{2n} +1|>2n -2n=0,
\end{equation*} by $|t|=1$,  using the triangle inequality and
\begin{equation*}
   |(2n-1)(\cos\omega+i\sin\omega)+1| =[(2n-1)^2+1+2(2n-1)\cos\omega]^{\frac 12}<2n
\end{equation*}  when $\omega\neq 2k\pi$.
  But if $t^{2n}=1$,
\begin{equation*}
  (2n-1)t^{2n}-2n t^{2n-1}+1 =  (2n-1) -2n t^{ -1}+1=2n(1-e^{ 2   i \theta})\neq0,
\end{equation*}
for $\theta\in ({3\pi\over 2},{2\pi} )$.  Thus
$
s(x_1+x_2 {\bf i} )\neq0$ for $x_1>0 $, $x_2<0 $. Consequently, by \eqref{eq:s-K}, $K(g) \not=0,$ for $  g\in \mathscr H^{n-1}\setminus \{0\}.$
\qed

\color{black}

\subsection{Proof of Theorem \ref{main3}}

We use the work of \cite{Str} on self-similar tilings to find a ``nice'' decomposition of $\mathscr H^{n-1}$, analogous to the decomposition of $\mathbb R^n$ into dyadic cubes in classical harmonic analysis.
For $\bm y=(y_1,\cdots, y_{4n-4})$, we use $|\bm y|_\infty$ to denote $\max\{ |y_1|, |y_2|, \dots , |y_{4n-4} | \}$, 
while $\mathscr H^{n-1}_{\mathbb Z}$ denotes the subgroup $\{ (\bm a,\bm b) \in \mathscr H^{n-1} \mid \bm a \in \mathbb Z^{4n-4}, \bm b\in \mathbb Z^3\}$.

Denote by $A$ the basic tile
\begin{align}\label{basic tile}
A:=\left\{(\bm t,\bm y)\in\mathscr H^{n-1}\mid \bm y\in [0,1]^{4n-4}, ~ F_j(\bm y)\leq t_j<F_j(\bm y)+1, j=1,2,3 \right\},
\end{align}
where
$$F_j(\bm y)=\sum_{m=1}^\infty{1\over 4^m}B_j\big( \left[ 2^m\bm y\right]~{\rm mod} ~2, \langle2^m\bm y\rangle\big),$$
here [ , ] and $\langle\, ,\rangle$ denote the integer and fractional part functions, interpreted componentwise  ($[y]_j=[y_j]$, etc.),
and $B_j(\bm y,\bm y') $ is defined in \eqref{law}, which is bounded on $[0,1]^{4n-4}\times [0,1]^{4n-4}$.

If $\gamma=(\bm a, \bm b)\in\mathscr H^{n-1}_{\mathbb Z}$ then the image $\tau_{\gamma}(A)$ of $A$ under left translation by $\gamma$ is
$$
\tau_{\gamma}(A)=\left\{(\bm t, \bm y)\mid \bm y-\bm a\in [0,1]^{4n-4}, 0\leq t_j-b_j-B_j(\bm a,\bm y-\bm a)-F_j(\bm y-\bm a)<1, j=1,2,3\right\}.
$$

According to \cite{Str}, $\mathscr H^{n-1}=\cup_{\gamma\in \mathscr H^{n-1}_{\mathbb Z}}~ \tau_{\gamma}(A)$ is a (disjoint) tiling of $\mathscr H^{n-1}$ and it is self-similar, i.e.
\begin{align}\label{A}
\delta_{2}(A)=\bigcup_{\gamma\in\Gamma_0}\tau_{\gamma}(A),\quad {\rm or}\quad
A=\bigcup_{\gamma\in\Gamma_0}\delta_{1\over 2}\circ\tau_{\gamma}(A),
\end{align}
where
$\Gamma_0=\left\{(\bm a, \bm b): a_j=0 ~{\rm or}~1, b_j=0, 1, 2,  ~{\rm or}~ 3\right\}.$

\begin{lem}\label{lem:inner}
The tile $A$ has inner points.
\end{lem}

\begin{proof}
Fix a positive integer $n_0$ such that $n_0>2+[\log_4  M]$. If $\bm y\in\mathbb R^{4n-4}$ satisfies 
  $0<|\bm y|_\infty<2^{-n_0}$, then
$0<2^n y_l<1$, $l=1,\cdots, 4n-4$, and $[2^n\bm y]=0$ for $n<n_0$.

Let $M=\max_{\bm y,\bm y'\in [0,1]^{4n-4}}\left\{| B_j(\bm y,\bm y')|, j=1,2,3\right\}$, then we have
\begin{align*}
|F_j(\bm y)|\leq\sum_{n=n_0}^{\infty}{M\over 4^n}<{M\over 4^{n_0-1}}<{1\over 4},\quad j=1,2,3.
\end{align*}
This implies that
$$F_j(\bm y)<{1\over 4}\quad{\rm and}\quad F_j(\bm y)+1>1-{1\over 4}={3\over 4},\quad j=1,2,3.$$
Therefore, from the definition of $A$ as in \eqref{basic tile}, we see that
$$(0, 2^{-n_0})^{4n-4}\times \Big({1\over 4},  {3\over 4}\Big)\subset A.$$
The lemma is proved.
\end{proof}

In what follows, we denote $g_{o} :=(2^{-n_0-1}, \ldots, 2^{-n_0-1}, {1\over 2})$ as the
``center'' of $A$.

\begin{defn}\label{def:tiles}
We define
\[
\tile_0 := \{ \tau_g(A) : g \in \mathscr H^{n-1}_{\mathbb Z} \},
\qquad
\tile_j := \delta_{2^j} \tile_0
\quad\text{and}\quad
\tile := \bigcup_{j \in \mathbb Z} \tile_j .
\]
We call the sets $T \in \tile$ \emph{tiles}.
If $j \in \mathbb Z$ and $g \in \mathscr H^{n-1}_{\mathbb Z}$ and $T = \delta_{2^j} \circ\tau_g(A)$, then $T = \tau_{\delta_{2^j} (g)}( \delta_{2^j} (A))$,
and we further define
\[
{\rm cent}(T) := \delta_{2^j} \circ\tau_g( g_{o}),
\qquad
{\rm width}(T) := 2^j.
\]
\end{defn}

\begin{lem}\label{thm:Heisenberg-grid}
Let $\tile_j$ and $\tile$ be defined as above.
Then the following hold:
\begin{enumerate}
  \item for each $j \in \mathbb Z$, $\tile_j$ is a partition of $\mathscr H^{n-1}$, that is, $\mathscr H^{n-1} = \bigcup_{T \in \tile_j} T$;
  \item $\tile$ is nested, that is, if $T, T' \in \tile$, then either $T$ and $T'$ are disjoint or one is a subset of the other;
  \item for each $j \in \mathbb Z$ and $T\in\tile_j$, $T$ is a union of $2^{4n+2}$ disjoint congruent subtiles in $\tile_{j-1}$;
  \item there exists $g\in T$, such that $B(g, C_1 q) \subseteq T \subseteq B(g, C_2 q)$, where  $q = {\rm width}(T)$ for each $T \in \tile$; the constants $C_1$ and $C_2$ depend only on $\cdim$;
  \item if $T \in \tile_j$, then $\tau_g(T) \in \tile_j$ for all $g \in \delta_{2^j} (\mathscr H^{n-1}_{\mathbb Z})$, and $\delta_{2^k} (T) \in \tile_{j+k}$ for all $k \in \mathbb Z$.
\end{enumerate}
\end{lem}

\begin{proof} It is clear that (1) holds.
(2) (3) (5) follow from \eqref{A}. (4) can be implied by Lemma \ref {lem:inner}.
\end{proof}

\color{black}
Every tile is a dilate and translate of the basic tile $A$, so all have similar geometry.
Hence each tile in $\tile_j$ has fractal boundary
and is ``approximately'' a quaternionic Heisenberg ball of radius $2^{j}$.
If two tiles in $\tile_j$ are ``horizontal neighbours'', then the distance between their centres is $2^{j}$, while if they are ``vertical neighbours'', then the distance is $2^{2j}$.

\begin{proof}[Proof of Theorem \ref{main3}]
Based on the construction of tiles, we see that for every $T\in \tile_{j}$ and for each fixed $N\in\mathbb{N}$, there exists a unique $T_{N+\mathfrak a_0}\in \tile_{N+j+\mathfrak a_0}$ such that $T\subset T_{N+\mathfrak a_0}$. Here $\mathfrak a_0$ is a  positive integer to be determined later.

We now fix $N\in\mathbb N$ and choose an arbitrary $T\in \tile_{j}$.
From Theorem \ref{main2}, we get that
$$ \quad K( g)\not=0,\qquad \ \forall  g\in \mathbb S^n,  $$
where $\mathbb S^n=\{g\in \mathscr H^{n-1}:\ \|g\|=1\}$ is the unit sphere in $\mathscr H^{n-1}$.

Since $K$ is a $C^\infty$ function in $\mathscr H^{n-1} \backslash \{0\}$, there exists $g_0$ in $\mathscr H^{n-1}$ with $\rho(g_0)=1$ such that
$$ |K(g_0)|=\min_{g\in \mathbb S^n } |K(g) |>0. $$

Hence, there exists $ 0<\varepsilon_1\ll1$ such that
\begin{align}\label{non zero e1}
 |K( g)|>{1\over 2} |K( g_0)|
\end{align}
for all $g\in U(\mathbb S^n , 4\varepsilon_1) = \{g\in\mathscr H^{n-1}: \exists \tilde g\in \mathbb S^n  {\rm\ \ such \ that} \ \ d(g,\tilde g)<4\varepsilon_1   \}$.

To continue, we first point out that for the chosen $T\in \tile_{j}$ and that unique tile $T_{N+\mathfrak a_0}\in \tile_{N+j+\mathfrak a_0}$ with $T\subset T_{N+\mathfrak a_0}$,
there exist $\hat h \in T_{N+\mathfrak a_0}$ with
$d(h, \hat h) = \mathfrak C 2^{N+j+\mathfrak a_0}$ and
 {$d(\hat h, T_{N+\mathfrak a_0}^c) >10C_2 2^j$, then $\tilde g_0 :=(\delta_{\mathfrak C^{-1}2^{-N-j-\mathfrak a_0}} (h^{-1} \cdot\hat h))^{-1} \in \mathbb S^n$}. Without lost of generality,
for $$  K(\tilde g_0)= K_1(\tilde g_0) +K_2(\tilde g_0) \textbf{i}+K_3(\tilde g_0)\textbf{j} +K_4(\tilde g_0)\textbf{k},  $$
we assume that {$|K_1(\tilde g_0)|\geq {1\over2} |K(\tilde g_0)|(> {1\over4} |K( g_0)|)$} and that $K_1(\tilde g_0)$ is positive. Then, there exists $0<\varepsilon_o<\varepsilon_1$ such that
\begin{align}\label{k1}
K_1(g)> {1\over4} |K( g_0)|,\quad g\in B(\tilde g_0, 4\varepsilon_o).
\end{align}

From the definition of $\tilde g_0 $ we see that
\begin{align}\label{g*}
 \hat h= h \cdot \delta_{\mathfrak C 2^{N+j+ \mathfrak a_0}}( \tilde g_0^{-1} ).
\end{align}
Next, we choose the integer $\mathfrak a_0$ so that $2^{N+\mathfrak a_0}> 5C_2\mathfrak C^{-1} \varepsilon_o^{-1}$. Then fix some $ \eta\in(0, 2 \varepsilon_o)$ such that
 the two balls $B(h,   \eta r)$ and $B( \hat h,   \eta r) $ with $r=\mathfrak C 2^{N+j+\mathfrak a_0}$ satisfy the following condition:
 {
$$5 C_2 2^j<  \eta r< 10 C_2 2^j.$$
}
Then we can deduce that $T\subset B(h,   \eta r)$ and $B( \hat h,   \eta r) \subset T_{N+\mathfrak a_0}$.

 It is direct that
 for every $g\in B(h,   \eta r)$, $\hat{g}\in B( \hat h,  \eta r)$, we can write
 $$ g =h\cdot \delta_r (g'_1), \quad  \hat{g} = \hat h\cdot \delta_r (g'_2), $$
 where $g'_1 \in B(0,   \eta)$, $g'_2 \in B(0,  \eta)$.

As a consequence, we have
\begin{align}\label{dilation}
K(g,\hat{g}) &= K\big( h\cdot \delta_r (g'_1) ,   \hat h\cdot \delta_r (g'_2)   \big)
= K\big(   h\cdot \delta_r (g'_1) ,   h \cdot \delta_r( \tilde g_0^{-1} )\cdot \delta_r (g'_2)   \big)\nonumber\\
&= K\big(    \delta_r (g'_1) ,     \delta_r( \tilde g_0^{-1} )\cdot \delta_r (g'_2)   \big)\\
&= K\big(    \delta_r (g'_1) ,     \delta_r( \tilde g_0^{-1} \cdot g'_2)   \big)\nonumber\\
&= r^{-Q} K\big(    g'_1,     \tilde g_0^{-1} \cdot g'_2  \big)\nonumber\\
&= r^{-Q} K\big(    (g'_2)^{-1} \cdot    \tilde g_0 \cdot g'_1  \big),\nonumber
\end{align}
where the second equality comes from \eqref{g*} and the third comes from the property of the left-invariance.

Next, we note that
\begin{align*}
d\big(    (g'_2)^{-1} \cdot    \tilde g_0 \cdot g'_1,\ \tilde g_0 \big) &= d\big(      \tilde g_0 \cdot g'_1,\  g'_2 \cdot  \tilde g_0 \big)
\leq  \, \left[ d\big(      \tilde g_0 \cdot g'_1,   \tilde g_0 \big)+ d\big(      \tilde g_0 , g'_2 \cdot  \tilde g_0 \big) \right]\\
&=  \, \left[  d\big(  g'_1,   0 \big)+ d\big( 0, g'_2 \big) \right]\\
&\leq 2  \eta\\
&<4  \varepsilon_o,
\end{align*}
which shows that $ (g'_2)^{-1} \cdot    \tilde g_0 \cdot g'_1$ is contained in the ball $B(\tilde g_0, 4 \varepsilon_o)$
for all $g'_1 \in B(0,  \eta)$ and for all $g'_2 \in B(0,  \eta)$.

Thus, from  \eqref{non zero e1} and \eqref{k1}, we obtain that
\begin{align}\label{lower bound e1}
| K\big(    (g'_2)^{-1} \cdot    \tilde g_0\cdot g'_1  \big)| > {1\over 2 } | K( g_0)|\quad {\rm with}
\quad K_1\big(    (g'_2)^{-1} \cdot    \tilde g_0 \cdot g'_1  \big)> {1\over 4 } | K( g_0)|>0,
\end{align}
 for all $g'_1 \in B(0,  \eta)$ and for all $g'_2 \in B(0,  \eta)$.

Now
combining the equality \eqref{dilation} and  inequality \eqref{lower bound e1} above, we obtain that
\begin{align}\label{lower bound e2}
|K(g,\hat{g})|   > {1\over 2 } r^{-Q} |K( g_0)| \quad {\rm with}\quad K_1(g,\hat{g}) > {1\over 4 } r^{-Q} |K( g_0)|
\end{align}
for every $g\in B(h,  \eta r)$ and for every $\hat g\in B( \hat h,  \eta r)$, where $K_1(g,\hat{g})$ and $K_1(\tilde g_0)$ have the same sign. Here $K(\tilde g_0)$ is a fixed
constant independent of $\eta$, $r$, $h$, $g_1$ and $g_2$. We denote
$$ C(n)= {1\over 2}|K( g_0)|.$$

From the lower bound \eqref{lower bound e2} above, we further obtain that
for the suitable $\eta\in (0,\varepsilon_o)$,
\begin{align}\label{lower bound e3}
|K(g,\hat{g})| > C(n) r^{-Q} \quad {\rm with}\quad K_1(g,\hat{g}) > {1\over 2 }  C(n)r^{-Q}
\end{align}
for every $g\in B(h,  \eta r)$ and for every $\hat{g}\in B( \hat h,  \eta r)$.

Based on the fact that $B( \hat h,  \eta r)\subset T_{N+\mathfrak a_0}$ and  {$\eta r> 5 C_2 2^j $}, there must be some tile $\hat T \in \tile_{j}$ such that $\hat{T}\subset B( \hat h,  \eta r)$. Also note that $T\subset B( h,  \eta r)$. Hence we obtain that  $\mathfrak a_{1} 2^{N+j}\leq d(\cent{(T)},\cent(\hat{T}))\leq \mathfrak a_{2}2^{N+j}$, where $\mathfrak a_1$ and $\mathfrak a_2$ depends only on $\mathfrak a_0$ and $\mathfrak C$. Moreover, we see that
  for all $(g,\hat g)\in T\times \hat{T}$, $K_1(g, \hat g)$ does not change sign and that
   for all $(g,\hat g)\in T\times \hat{T}$, $|K_1(g, \hat g)|\gtrsim  2^{-Q(N+j)}$, where the implicit constant depends on $C(n)$ and $\mathfrak a_0$.

The proof of Theorem \ref{main3} is complete.
\end{proof}
\color{black}

\section{regular functions and  heat kernel integrals on  the quaternionic Heisenberg group}

\subsection{Subharmonicity }

In the early attempts to generalize $H^p$ to several
  variables, Stein    and Weiss \cite{SW} considered a $\mathbb{R}^n$-valued function
$u = (u_1, u_2, . . . , u_n)$   on $\mathbb{R}_+^n$ that satisfies the   generalized Cauchy--Riemann equations
\begin{equation}\label{eq:GCR0}
   \frac {\partial u_j}{\partial x_k}= \frac {\partial u_k}{\partial x_j} \quad {\rm and}\quad \sum_{k=1}^n\frac {\partial u_k}{\partial x_k}=0,
\end{equation}
 which implies that $u$ is harmonic and $|u(x)|^p$ is subharmonic for $p\geq\frac {n-2}{n-1}$. Recall that for a domain $\Omega\subset \mathbb{R}^n$, an upper semicontinuous function $v :\Omega\rightarrow
 [-\infty,\infty)$
is {\it  subharmonic} if $\triangle v\geq 0$ in the sense
of distributions, where $\triangle$ is the Laplace operator.
 Based on the subharmonicity, they  built an $H^p$ theory for the generalized  Cauchy--Riemann systems. It is natural to try to get below $p=\frac {n-2}{n-1}$, and this can be done by studying higher
gradients of harmonic functions in place of (\ref{eq:GCR0}) by Calder\'on-Zygmund \cite{CZ}. So there is no restriction for real variable theory of the $H^p$ spaces, but such restriction is natural for the $H^p$ space of functions satisfying the   generalized Cauchy--Riemann equations.

The subharmonicity of   solutions to the Cauchy--Fueter equation on $\mathbb{R}^4$ was mentioned in
\cite[p.164]{SW} without proof. Let us recall Stein-Weiss general results \cite{SW}.
Let $U,V$ be two finite dimensional complex vector spaces which are irreducible representations of ${\rm Spin}(n)$, the covering group ${\rm SO}(n)$. $V$ is $n$-dimensional with a distinguished real $n$-dimensional subspace $V_0$, identified with $\mathbb{R}^n$. For $u\in C^1(\Omega,U)$ for a domain $\Omega$ in $\mathbb{R}^n$, the gradient of $u$, $\nabla u(x)$ for each $x\in \Omega$, is valued in $U \otimes V$. $U \otimes V$ is irreducible and can be decomposed as
\begin{equation*}
   U \otimes V= U \boxtimes V\oplus  (U \boxtimes V)^\perp
\end{equation*}
where $U \boxtimes V$ is the {\it  Cartan composition}, and $(U \boxtimes V)^\perp$ is its orthogonal complement. Then, we can decompose
\begin{equation*}
   \nabla u(x)=\partial u(x)+\overline{\partial} u(x)
\end{equation*}orthogonally,
and $\overline{\partial} $ is called the {\it generalized Cauchy--Riemann operator}.

\begin{thm} \label{thm:subharmonic} \cite[Theorem 1]{SW} If $u$ is a solution of $\overline{\partial} u=0$, then $u$ is harmonic and $|u(x)|^p$ is subharmonic for $p\geq\frac {n-2}{n-1}$.
\end{thm}

\begin{cor} \label{cor:subharmonic} For any solution  $F$  to the Cauchy--Fueter equation on a domain in $\mathbb{R}^4$,  $|F(x)|^p$ is subharmonic for $p\geq\frac {2}{3}$.
\end{cor}
\begin{proof}
   Let us check that the Cauchy--Fueter operator on   $\mathbb{R}^4$ is the   generalized Cauchy--Riemann operator for  {$ {U}=\mathbb{C}^2$}, a spin representation of ${\rm Spin}(4)$. In \cite[Section 8]{SW}, Stein-Weiss discussed all such operators on $\mathbb{R}^4$. Then
$\frac {2}{3}$ is  {best one} for the Cauchy--Fueter operator by \cite[Theorem 4]{SW}.

Assume a unitary representation $\zeta\rightarrow R_\zeta$ of ${\rm Spin}(n)$ acts on $U$.
Suppose that $\{f_\alpha\}$ is an orthonormal  {basis of $U$} and $\{e_j\}$ is an orthonormal basis of $V_0$. We can write $ u(x)=\sum u_\alpha(x) f_\alpha$. Then $ U \otimes V$ has a basis $\{f_\alpha\otimes e_j\}$, and
\begin{equation*}
   \nabla u(x)=\sum\frac {\partial u_\alpha}{\partial x_j}(x) f_\alpha\otimes e_j.
\end{equation*}
Recall \cite[Section 2]{SW} that if $v(x):=R_\zeta[u(\rho_{\zeta^{-1}}(x))]$, where $\zeta\rightarrow \rho_\zeta$ is the rotation representation on $V_0$, then
\begin{equation}\label{eq:rep-gradient}
  \nabla v(0):=(R_\zeta \otimes\rho_\zeta) (\nabla u)(0) .
\end{equation}

Now let  $V$ be $ \mathbb{C}^4$ with a  {distinguished} real $4$-dimensional subspace given by the embedding of  $\mathbb H$ into the space of complex $2\times2$ matrices $\mathbb{C}^{2\times 2 } \simeq \mathbb{C}^4 $
\begin{equation}\label{eq:embed-Q1}q=x_1+  x_2\mathbf{i}+ x_3\mathbf{j}
+x_4\mathbf{k}  \mapsto
\tau(q)
	=	\begin{pmatrix}
	x_1+ix_4&-x_2-ix_3\\
x_2 -ix_3&x_1-ix_4
	\end{pmatrix}.
\end{equation} $\tau$ is a representation, i.e. $\tau(q_1q_2)	=\tau(q_1)\tau(q_2)	$ (cf. \cite{wang-alg}).	
We can identify ${\rm SU}(2)$ with unit quaternions. It is well known
   ${\rm Spin}(4)\cong {\rm SU}(2)\times  {\rm SU}(2)$, which acts on $V_0$ as
   $\rho(q_1,q_2)\tau(q)=\tau(q_1)\tau(q)\tau(\overline{q_2 })$.
We denote    ${\rm Spin}(4)\cong {\rm SU}(2)_L\times  {\rm SU}(2)_R$, where ${\rm SU}(2)_L$ and $ {\rm SU}(2)_R$ acts on $\mathbb{C}^2$ naturally,  {denoted by} $\mathbb{C}^2_L$ and $\mathbb{C}^2_R$,  {as}  $2$-dimensional column and row vectors, respectively. They are two  spin representations of ${\rm Spin}(4)$. $\mathbb{C}^4 $  {as the} space of $2\times2$ matrices means
\begin{equation*}
   \mathbb{C}^4\cong \mathbb{C}^2_L\otimes\mathbb{C}^2_R.
\end{equation*}

For $n=4$ and $U=\mathbb{C}^2_R$, we see that $(\frac {\partial u_\alpha}{\partial x_j}(x))$   is in the representation $\mathbb{C}^2_R\otimes \mathbb{R}^4$ by (\ref{eq:rep-gradient}). Under action $\tau$ in (\ref{eq:embed-Q1}), we see that $\tau(\frac {\partial u_\alpha}{\partial x_j}(x))\in \mathbb{C}_R^2\otimes  \mathbb{C}^2_L\otimes\mathbb{C}^2_R\cong \mathbb{C}_R^2\otimes\mathbb{C}^2_R\otimes  \mathbb{C}^2_L$
given by
\begin{equation*}
   \nabla u(x)=\sum_{\alpha,\beta,\mu=0,1} \partial_{\mu\beta}u_\alpha(x)f_\alpha\otimes f_{\beta}\otimes g_\mu
\end{equation*}
 if $\{f_\alpha\}$ is an orthonormal basis of $\mathbb{C}_R^2 $ and $\{g_\mu\}$ is an orthonormal basis of $\mathbb{C}_L^2$, and we denote
\begin{equation} (\partial_{ \mu\beta}):=\begin{pmatrix}
	\partial_{x_1}+i\partial_{x_2}&-\partial_{x_3}-i\partial_{x_4}\\
\partial_{x_3} -i\partial_{x_4}&\partial_{x_1}-i\partial_{x_2}
	\end{pmatrix},
\end{equation}

$\mathbb{C}_R^2\otimes \mathbb{C}^4$ decomposes into irreducible ones as
 \begin{equation*}\begin{split}
   \mathbb{C}_R^2\otimes \mathbb{C}^4&\cong \mathbb{C}_R^2\otimes  \mathbb{C}^2_L\otimes\mathbb{C}^2_R \cong \mathbb{C}_R^2\otimes \mathbb{C}^2_R\otimes  \mathbb{C}^2_L
   \cong \left(\odot^2\mathbb{C}^2_R \oplus  \Lambda^2\mathbb{C}^2_R\right)\otimes\mathbb{C}^2_L \cong \left (\odot^2\mathbb{C}^2_R\otimes\mathbb{C}^2_L\right)\oplus  \mathbb{C}^2_L.
\end{split}\end{equation*}
This symmetric and antisymmetric parts decomposition is realized as
\begin{equation*}
   f_\alpha\otimes f_{\beta}=f_\alpha\odot f_{\beta}+ f_\alpha\wedge f_{\beta},
  \end{equation*}
 where, $ f_\alpha\odot f_{\beta}:=(f_\alpha\otimes f_{\beta}+f_\beta\otimes f_{\alpha})/2$, $f_\alpha\wedge f_{\beta}:= (f_\alpha\otimes f_{\beta}-f_\beta\otimes f_{\alpha})/2$.
Thus, the projection of $\nabla u(x)$ to $   \mathbb{C}^2_L$ is simply
\begin{equation*}
    \overline{\partial} u(x)=\sum_{\mu=0,1} ( {\partial} _{\mu 1 }u_0 (x)- {\partial}_{ \mu 0} u_1 (x)) f_0\wedge f_{1}\otimes g_\mu.
\end{equation*}
As a result, the   generalized Cauchy--Riemann $\overline{\partial} u(x)=0$ for $U=\mathbb{C}^2_R$ is equivalent to
\begin{equation}\label{eq:dbar}
   \overline{\partial} _{ \mu 1}u_0 (x)-\overline{\partial}_{ \mu 0} u_1 (x)=0,\qquad \mu=0,1.
\end{equation}

On the other hand, apply the representation $\tau$ (\ref{eq:embed-Q1}) to the Cauchy--Fueter $
\overline\partial_{{q} }F=0$  to get
\begin{equation}\label{eq:GCR}
   \begin{pmatrix}
	{\partial} _{0 0} &{\partial} _{01 }\\
{\partial} _{1 0 }&{\partial} _{1 1 }
	\end{pmatrix}\begin{pmatrix}
	\widetilde{F}_0&-\overline{\widetilde{F}}_1\\
\widetilde{F}_1&\overline{\widetilde{F}}_0
	\end{pmatrix}=0,
\end{equation}
where $\widetilde{F}_0=F_{ { 1}}+
 \mathbf{i}F_{ { 2}},  \widetilde{F}_1=F_{ { 3}}- \mathbf{i}F_{ { 4}} $. From the second column of this matrices equation, we see that $u_1= \overline{\widetilde{F}}_1$ and $u_0= \overline{\widetilde{F}_0}$ satisfy the   generalized Cauchy--Riemann (\ref{eq:dbar}) for $U=\mathbb{C}^2_R$. Thus, by Theorem
\ref{thm:subharmonic}, $|u|^p=|F|^p$ is subharmonic for $p\geq\frac {2}{3}$.
\end{proof}

 \begin{prop}\label{lem:bound}
    There is a positive constant $C$
  such that for all $\frac 23\leq p \leq 1$ and $(\varepsilon , \xi)\in \mathscr U$, we have
  \begin{equation}\label{eq:bound}
    |f ( \varepsilon , \xi)|\leq  C\|f \|_{H^p( \mathscr U ) }\varepsilon^{-\frac {2n+1}p}
  \end{equation}
  for any $f\in H^p( \mathscr U ) $.
 \end{prop}
 \begin{proof}
 Note that if $f\in H^p( \mathscr U )$, then $\widetilde{f}(\varepsilon, g')=f(\varepsilon,g\cdot g')$ (for a fixed $g$) is also regular, because
\begin{equation*}
   \overline Q_l\widetilde{f}( \varepsilon ,g'  )= \overline Q_l f( \varepsilon ,g\cdot g'  )=0
\end{equation*} due to the left invariance of $Y_j$'s and the fact that $f$ is regular. Moreover, we have
\begin{equation*}
   \int_{\mathscr H^{n-1} }|\widetilde{f}( \varepsilon , g' )|^pdg'=
\int_{\mathscr H^{n-1} }| {f}( \varepsilon , g'  )|^pdg'
\end{equation*}
by the invariance of the measure. We get $\widetilde{f}\in H^p( \mathscr U )$. So it is sufficient to prove (\ref{eq:bound}) for $g=0$.

Recall that
   $F(q_1  , q' )=f(q_1-|q'|^2 , q')$ is an element of $ H^p(\mathcal{U} )$ with the same norm by (\ref{eq:isomorphism}).   By Corollary \ref{cor:subharmonic},  $|F|^p$ is   subharmonic on $ \mathcal{ U }$ for each quaternionic variable. Consequently, the submean value inequality holds for each quaternionic variable. By
 the  subharmonicity of $|F|^p(q_1,0,\ldots,0)$,  we get
   \begin{equation*}
     |F|^p(\varepsilon,0)\leq \frac 1{| B_{\mathbb H}(\varepsilon,\varepsilon/2)|} \int_{ B_{\mathbb H}(\varepsilon,\varepsilon/2)}|F|^p(q_1,0,\ldots,0)dV(q_1),
   \end{equation*}
    where  $ B_{\mathbb H}(x,r)$ is a ball in $\mathbb{H}$ with radius $r$ and center $x$. Then apply the submean value inequality to subharmonic function $|F|^p(q_1,q_2,0,\ldots,0)$ in variable $q_2$ to get  \begin{equation*}
     |F|^p(\varepsilon,0)\leq \frac 1{| B_{\mathbb H}(\varepsilon,\varepsilon/2)\times  B_{\mathbb H}(0,\sqrt\varepsilon/2)|} \int_{ B_{\mathbb H}(\varepsilon,\varepsilon/2)} dV(q_1)\int_{   B_{\mathbb H}(0,\sqrt\varepsilon/2)}|F|^p(q_1,q_2,0,\ldots,0)dV(q_2).
   \end{equation*}  Repeating this procedure, we finally get
   \begin{equation*}
     |F|^p(\varepsilon,0)\leq \frac 1{|D|} \int_D|F|^p(q)dV(q),
   \end{equation*}
  where $D= B_{\mathbb H}(\varepsilon,\varepsilon/2)\times  B_{\mathbb H}(0,\sqrt\varepsilon/2)\times \cdots\times  B_{\mathbb H}(0,\sqrt\varepsilon/2)$. Noting that $D\subset\{q \in \mathcal{U}: \varepsilon/4< \operatorname{Re} q_1-|q'|^2< 3\varepsilon/2 \}$, we have
   \begin{align*}
     |F|^p(\varepsilon,0)&\leq \frac 1{|D|}    \int_{\{  \varepsilon/4< \operatorname{Re} q_1-|q'|^2< 3\varepsilon/2 \}}  \left|F\left(x_1 ,x_2\cdots ,  x_{4n}\right)\right|^p dx_1dx_2\cdots dx_{4n}
      \\& \leq \frac {2^{4n}}{\varepsilon^{2n+2}}    \int_{(\frac {\varepsilon}4,\frac {3\varepsilon}2)\times \mathbb{R}^{4n-1}}  \left|F\bigg(x_1+\sum_{j=5}^{4n}|x_j|^2,x_2\cdots ,  x_{4n}\bigg)\right|^p dx_1dx_2\cdots dx_{4n}
      \\
    & \leq \frac {2^{4n}}{\varepsilon^{2n+2}}  \int_{\varepsilon/4}^{  3\varepsilon/2 }dx_1\int_{\partial\mathcal{U} } |F (p+ x_1  \mathbf{e}_1)|^p d\beta(p)
     \leq \frac {5\cdot 2^{4n-2}}{\varepsilon^{2n+1}}\|F\|_{H^p(\mathcal U )}^p,
\end{align*}
as in \cite[(4.5)]{CMW},
  where    we have used the coordinates transformation
  \begin{equation*}
     (x_1 ,\cdots ,  x_{4n} )\rightarrow\bigg(x_1+\sum_{j=5}^{4n}|x_j|^2,x_2\cdots ,  x_{4n}\bigg),
  \end{equation*}
    whose  Jacobian is the identity. The estimate follows from $ F  (\varepsilon,0)=f(\varepsilon,0)$ and $\|F\|_{H^p(\mathcal U )}=\|f\|_{H^p(\mathscr U )}$.
  \end{proof}

\subsection{The heat equation }
 Recall that it is a fundamental fact in  $H^1$ theory that an $H^1$ harmonic function     on $\mathbb{R}^{n+1}_+$  can be expressed as Poisson integral of its boundary value. To develop $H^1$ theory for holomorphic  functions on a strongly pseudoconvex domain ${\mathcal D }$, there is a natural candidate, the {\it Poisson-Szeg\H o kernel}, which is the function $\mathscr P(z, w)$ on ${\mathcal D }\times \partial{\mathcal D }$ defined by
\begin{equation}\label{eq:Poisson-Szego}
   \mathscr P(z, w)=\frac {|S(z, w)|^2}{ {S(z, z)}}, \qquad z\in {\mathcal D }, w\in \partial{\mathcal D },
\end{equation}if $S(z, z)\neq 0$. Here $S(z, w)$ is the {\it Cauchy--Szeg\H o kernel} reproducing $H^2(\mathcal D )$ functions, which is holomorphic in $z$ and anti-holomorphic in $w$.
$P$ reproduces
$H^2(\mathcal D )$ functions, because, for $f\in H^2(\mathcal D )$,
\begin{equation*}
   \int_{{\mathcal S}} \mathscr P(z, w)f(w)d\beta(w)=\frac {1}{ {S(z, z)}} \int_{\partial{\mathcal D }} S(z, w) \overline{S(z, w)}f(w)d\beta(w)=f(z)
\end{equation*}
by Cauchy--Szeg\H o kernel   reproducing $H^2(\mathcal D )$ function  $\overline{S(z, \cdot)}f(\cdot) $ for fixed $z$ and $ {S(z, z)} $ real.
We have
$
   \int_{\partial{\mathcal D }} \mathscr P(z, w) d\beta(w)=1.
$ Note that Cauchy--Szeg\H o kernel restricted to the boundary is a singular integral, and so it is not an approximation to the identity.
The Poisson-Szeg\H o  reproducing formula was used by Koranyi \cite{Ko} and
Garnett-Latter  \cite{GL} to study holomorphic $H^1$ space over the unit ball in $\mathbb{C}^n$ and its atomic decomposition. This construction does not work in the noncommutative case, because   $\overline{S(z, \cdot)}f(\cdot)$ is not regular in general although $\overline{S(z, \cdot)}$ and $f(\cdot)$ are both regular. For example $x_1+\mathbf{i}x_2$ and $x_2+\mathbf{j}x_4$ are both annihilated by the Cauchy--Fueter operator $\overline{\partial}_{q}:={\partial}_{x_{1}}+\textbf{i}{\partial}_{x_{2}}+
\textbf{j}{\partial}_{x_{3}}+\textbf{k}
{\partial}_{x_{4}} $, but their product $(x_1+\mathbf{i}x_2)(x_2+\mathbf{j}x_4)$ is not.

In early 1970s, little information of  Cauchy--Szeg\H o kernel for a strongly pseudoconvex domain  in  $\mathbb{C}^n$ was known. Because   real and imaginary parts of a holomorphic function are both harmonic, Stein \cite{St} used the Euclidean Poisson integral in a clever way to prove the admissible convergence almost everywhere of a bounded holomorphic function with approach regions defined in terms of Carnot--Carath\'eodory (C-C) distance on the boundary. This technique was used to control     the (non)tangential maximal function
of an $H^p$ function on some bounded pseudoconvex domains in  $\mathbb{C}^n$ by Krantz--Li \cite{KL1} and Dafni \cite{Da}. But in the unbounded case,   Poisson integral is an integral with respect to the surface measure of the boundary, which blows up at infinity with respect to the Lebegue measure on the Heisenberg group. Moreover, since the Lipschitz constant of the boundary is unbounded, the detail of Poisson integral is not known directly. The key point is to find a reproducing formula that matches the geometry of the boundary.

On the other hand, since  a holomorphic function is annihilated by the Beltrami--Laplace operator associated to a K\"ahler metric,  potential theory associated to this operator (the    Dirichlet problem and its solution, etc.)  gives us information  of  holomorphic   function on the domain. In general, we need to choose a K\"ahler metric matching  C-C geometry on the boundary.   Geller \cite{Gel} wrote down the Beltrami--Laplace operator associated to the complex hyperbolic metric on the   Siegel upper half space, and found the solution to the corresponding  Dirichlet problem explicitly. It was  generalized by  Graham \cite{Gra} to some modifications of  the Beltrami--Laplace equation. This solution formula   reproduces  holomorphic   functions,  playing the role of Poisson integral,  and  can be used to prove the boundary value of a  holomorphic $H^1$ function  on the   Siegel upper half space belongs to the boundary Hardy space $H^1$ over the Heisenberg group \cite{Gel}.

\begin{proof}[Proof of Theorem \ref{prop:regular-sublap}]
By (\ref{eq:b})---(\ref{eq:Y-bracket}), we have
  \begin{align*}
Q_{l+1 } \overline{Q}_{l+1 } = &(Y_{4 l +1}-
 \mathbf{i}Y_{4 l +2}
 -\mathbf{j}Y_{4 l +3}-\mathbf{k}Y_{4 l +4})( Y_{4 l +1}+
 \mathbf{i}Y_{4 l +2}
 + \mathbf{j}Y_{4 l  +3}+ \mathbf{k}Y_{4 l +4})\\
= &\sum_{j=1}^4 Y_{4 l +j}^2+\mathbf{i}([Y_{4 l +1},
Y_{4 l +2}]-[Y_{4 l  +3}, Y_{4 l +4}] )+\mathbf{j}([Y_{4 l +1},
Y_{4 l +3}]+[Y_{4 l +2}, Y_{4 l +4}] )\\&\qquad\qquad+\mathbf{k}([Y_{4 l +1},
Y_{4 l +4}]-[Y_{4 l +2}, Y_{4 l +3}] )\\
= &\sum_{j=1}^4 Y_{4 l +j}^2+4\mathbf{i}\sum_{\alpha=1}^3 (b_{12}^\alpha-b_{34}^\alpha)   \partial_{ t_\alpha}+4\mathbf{j}\sum_{\alpha=1}^3 (b_{13}^\alpha+b_{24}^\alpha)   \partial_{ t_\alpha} +4\mathbf{k}\sum_{\alpha=1}^3 (b_{14}^\alpha-b_{23}^\alpha)   \partial_{ t_\alpha} \\=&\sum_{j=1}^4 Y_{4 l +j}^2+ 8\left(
\textbf{i}\partial_{t_{1}}+
\textbf{j}\partial_{t_{2}}+\textbf{k}\partial_{t_{3}}\right),
  \end{align*}
 for $l=0,\cdots,n-2$. Hence if $f$ is regular, we have
 \begin{equation*}\begin{split}
  ( \triangle_H+8(n-1) \partial_{t })f= \left(-\sum_{l=0}^{n-2} Q_{l+1 }\overline{Q}_{l+1 } + 8(n-1) \overline Q_0\right)f=0
  \end{split}\end{equation*}
  by $\overline Q_jf=0$ by Proposition \ref{prop:regular-equiv}.
 \end{proof}
 We can use heat semigroup $e^{-\frac t{8(n-1)}\triangle_H}$ to define the Littlewood--Paley function   and the Lusin area integral to study boundary behavior of regular  functions on $\mathscr U$.

 \begin{rem} (1) In \cite{Gel}, Geller developed the $H^1$ theory for holomorphic functions on the Siegel upper half space $\mathcal{D}:=\{(z',z_{n+1})\in \mathbb{ C}^n\times \mathbb{C}^1\mid \rho=\operatorname{Im} z_{n+1}-|z'|^2\}$, by using the   Laplace--Beltrami operator
   \begin{equation}\label{eq:Beltrami}
     \triangle_B=4\rho\left(\sum_{j=1}^n  Z_{ j}\overline{Z}_{ j}+4\rho Z_{n+1}\overline{Z}_{n+1}+2\mathbf{i}n \overline{Z}_{n+1}\right)=4\rho\left(\frac 14\sum_{j=1}^n  (X_{ j}^2+Y_{ j}^2)+ \partial _{\rho}^2+\partial_{ s}^2 -n\partial_ \rho\right)
  \end{equation}
  with respect to  the Bergman metric
 (cf. \cite[(2.2)]{Gra}), which annihilates holomorphic functions ($\overline{Z}_jf=0$), where $Z_j=\frac 12(X_j-\mathbf{i}Y_j)$ with $X_j=\partial_{x_j}+2y_j\partial_s$ and $Y_j=\partial_{y_j}-2x_j\partial_s$ and $Z_{n+1}=\frac 12(\partial_s-\mathbf{i}\partial_\rho)$.

 (2)
  If we throw out the term $4\rho Z_{n+1}\overline{Z}_{n+1}$ in (\ref{eq:Beltrami}),  holomorphic functions are annihilated by  the heat operator $\sum_{j=1}^n  (X_{ j}^2+Y_{ j}^2)  -4n\partial_\rho$. This fact seems to not have been noticed before in literatures, and can be used to simplify proofs in \cite{Gel} and may be generalized to other domains.
 \end{rem}

 \subsection{The heat kernel integral}

 Denote
 \begin{equation*}
    \mathcal{L}:= \frac 1{8(n-1)}\sum_{j=1}^{4 n-4}  Y_j^2-\partial_{t }.
 \end{equation*}We need the following   maximum principle for heat equation. Although there is a  maximum principle for heat equation on  any
 homogeneous group  for smooth functions  in \cite[Proposition 8.1]{FoSt}, we give its proof here since we need the proof for nonsmooth functions later.
 \begin{prop}  \label{prop:max}  {\rm (Maximum principle)} Let $D$ be a bounded domain in $\mathscr H^{n-1}$ and $\Omega=(0,T)\times D $ for $T>0$. Suppose that
 $v\in C^2( \overline{{\Omega}}) $,  $v|_{[0,T)\times \partial D }\leq 0$, $v|_{ \{0\}\times D }\leq 0$ and $\mathcal{L}v\geq 0$ in $\Omega$. Then $v\leq 0$ in $\Omega$.
\end{prop}
\begin{proof} If   replace $v$ by $v-\kappa_1t-\kappa_2$ with $ \kappa_1,\kappa_2>0$, we may assume $v|_{[0,T)\times\partial\Omega }< 0$, $v|_{ \{0\}\times\Omega }<0$ and $\mathcal{L}v> 0$.

Suppose that $v>0$ somewhere in $ {\Omega}$. Let $t^*:=\inf\{t\mid v(t,g)>0$ for some $g\in \overline{\Omega}\}$. By continuity
and $v$ negative on the boundary $[0,T)\times\partial\Omega\cup  \{0\}\times\Omega $, we see that there exists $(t^*, g^*) \in \Omega$
such that $v(t^*, g^*)=0$ and $t^*>0$. We must have $v(t,g )<0$ for $0<t<t^*$, $g\in\Omega$, and so $\partial_tv(t^*, g^*)\geq 0$.
On the other hand, $v(t^*,\cdot ) $ attains its maximum at $g^*$, which implies that $Y_jv(t^*, g^*)=0$ and
\begin{equation*}
   Y_j^2v(t^*, g^*)=\left.\frac {d^2}{ds^2} v(t^*,g^*(\ldots,0,s,0,\ldots) )\right|_{s=0}\leq 0,
\end{equation*}$j=1,\ldots,4n-4$, where $s$ appears in the $j$-th entry. Consequently, we get $\mathcal{L}v(t^*, g^*)\leq 0$, which contradicts  to $\mathcal{L}v> 0$ in $\Omega$. Thus $v-\kappa_1t-\kappa_2\leq 0$. Now letting $\kappa_1,\kappa_2\rightarrow 0+$, we get the result.
 \end{proof}

  \begin{prop}  \label{prop:heat-kernel-H1}  Suppose $f\in H^1(\mathscr U )$ and $h_t(g)$ is the heat kernel   $e^{-\frac t{8(n-1)}\triangle_H}$. Then (\ref{eq:heat-rep}) holds.
\end{prop}
\begin{proof}
 Since $f_{(k)}(t ,g):=f (t+1/k,g) $ is smooth and satisfies the heat equation (\ref{eq:regular-sublap}) for   $ t\geq 0, g \in  \mathscr H^{n-1}$,
 and $f ( 1/k,\cdot)\in L^1(\mathscr H^{n-1})$, let
 \begin{equation*}
    \widehat{f}_{(k)}( t, g)=\int_{\mathscr H^{n-1}}h_t(g'^{-1}\cdot g)f \Big( \frac{1}{k},g'\Big)dg',
 \end{equation*}
 which is smooth in $\mathscr U$ by the heat kernel estimate  \cite[Theorem IV 4.2]{VSC},  and satisfies
   $ \mathcal{L}\widehat{f}_{(k)}=0$. To apply  maximum principle to each components, write $f_{(k)}=f_ {(k);1} +\mathbf{i }f_{(k);2} +\mathbf{j }f_{(k);3} +\mathbf{k} f_{(k);4}, $ $\widehat{f}_{(k)}=\widehat{f}_{(k);1} +\mathbf{i }\widehat{f}_{(k);2} +\mathbf{j }\widehat{f}_{(k);3 } +\mathbf{k} \widehat{f}_{(k);4 } $. Then,
 \begin{equation*}
    \mathcal{L}\left(f_{(k);j }-\widehat{f}_{(k);j }\right)=0\quad {\rm and} \quad \left.\left(f_{(k);j}-\widehat{f}_{(k);j}\right)\right|_{\{0\}\times \mathscr H^{n-1}}=0,
 \end{equation*} $j=1,2,3,4$.
 We claim that for given $T>0$ and $\varepsilon>0$, there exists $R>0$ such that
 \begin{equation}\label{eq:claim}
    |f_{(k);j}-\widehat{f}_{(k);j}|\leq \varepsilon \quad {\rm on} \quad  [0,T)\times \partial B(0,R).
 \end{equation}
   Then we can apply the maximum principle (Proposition \ref{prop:max}) to $f_{(k);j}-\widehat{f}_{(k);j}- \varepsilon$ to get $f_{(k);j}-\widehat{f}_{(k);j}\leq \varepsilon$. Now let $R\rightarrow\infty$ and $T\rightarrow\infty$, we get $f_{(k);j}\leq \widehat f_{(k);j}$ on $\mathscr U $. The same argument gives us $\widehat f_{(k);j} \leq {f}_{(k);j} $ on $  \mathscr U $. Hence $f_{(k)}= \widehat{f}_{(k)}$ on $ \mathscr U $, i.e.
 \begin{equation*}
    {f} \left( t+\frac 1k, g\right)=\int_{\mathscr H^{n-1}}h_t(g'^{-1}\cdot g)f \left( \frac 1k,g'\right)dg'.
 \end{equation*}
  Then letting $k\rightarrow +\infty$, we get
(\ref{eq:heat-rep}).

To prove the claim (\ref{eq:claim}), note that
\begin{equation}\label{eq:estimate1}
  \left |\int_{\mathscr H^{n-1}}h_t(g'^{-1}\cdot g)f \left ( \frac 1k, g'\right)dg'\right|\leq{\varepsilon\over 2}
\end{equation}by the heat kernel estimate again,
for $\|g\|\geq R$ and $t>\frac 1k$ for sufficiently large $R$, since $f  ( \frac 1k, \cdot )\in L^1(\mathscr H^{n-1})$.
For the estimate for $f_{(k)}$, we need to use the mean value formula for $\pi^* f$, which  is regular on $\mathcal U $  and
\begin{equation*}
   {\tau}_{\pi^{-1}(g_0)}^*\circ \pi^* f= \pi^*\circ\tau_{g_0}^* f
\end{equation*}
by definition of $\pi$ in (\ref{eq:pi}), where $ {\tau}_{\pi^{-1}(g_0)}$ is the corresponding translation of $\mathcal U $ with $ \pi^{-1}(g_0)\in\partial \mathcal U $.
Then given $\varepsilon>0$, if $R$ is sufficiently large, we have
\begin{equation}\label{eq:small}
   \int_{\|g_0^{-1}\cdot g\|<\frac 1k, |t-t_0|<\frac 1{2k}}|f( t, g)| dt dg \leq \int_{\|  h \|\geq R,  t\in(\frac 1{2k},T+\frac  2{k})}|f( t,g)| dt d g \leq\varepsilon,
\end{equation}
for $\|g_0\|>R+1,  t_0\in (\frac 1{ k},T +\frac 1{ k})$,
since $f\in L^1((\frac 1{2k},T+\frac 2{k})\times\mathscr H^{n-1})$ by $f\in H^1(\mathscr U )$. Note that
$\pi^*\circ\tau_{g_0}^* f$ is regular in the product $\Omega_k:=B_{\mathbb{H}}(t_0,\delta)\times B_{\mathbb{H}}(0,\delta)\times\ldots\times B_{\mathbb{H}}(0,\delta)\subset
\mathcal{U}_n$ for some $\delta>0$ only depending on $k$, satisfying
\begin{equation*}
\pi(\Omega_k)\subset \left\{ ( t,g  )\mid \| g\|<\frac 1k, |t-t_0|<\frac 1{2k}\right\}\subset \mathscr{U},
\end{equation*}
  where $B_{\mathbb{H}}(q,r)$ is a ball in $\mathbb{H}$. Hence,
\begin{align}\label{eq:estimate2}
  | f( t_0,g_0)|&=\left|\pi^*\circ\tau_{g_0}^* f( t_0,0)\right|
  =\frac 1{|\Omega_k|}\left|\int_{\Omega_k}\pi^*\circ\tau_{g_0}^*f( t,g)dt  dg \right|\\&
  \leq \frac 1{|\Omega_k|}\left|\int_{\|g \|<\frac 1k,|t-t_0|<\frac 1{2k} }\tau_{g_0}^* f( t,g) dt dg\right|\nonumber\\
  &
  \leq \frac 1{|\Omega_k|}\left|\int_{\|g_0^{-1}\cdot g\|<\frac 1k, |t-t_0|<\frac 1{2k} } f( t,g)dt dg \right|\nonumber\\
  & \leq{\varepsilon\over|\Omega_k|},\nonumber
  \end{align}
  for $\|g_0\|>R+1,  t_0\in (\frac 1{ k},T +\frac 1{ k})$, by (\ref{eq:small}).
The claim follows from estimates (\ref{eq:estimate1}) and (\ref{eq:estimate2}).
  \end{proof}

 \begin{prop}  \label{prop:heat-kernel-Hp}  Suppose $f\in H^p(\mathscr U )$  for   $\frac {2}{3} < p \leq 1$ and continuous on $\overline{\mathscr U }$.  Then for $\frac {2}{3} < q\leq p  $, we have
  \begin{equation}\label{eq:heat-rep-p}
   | f ( t,g )|^q \leq \int_{\mathscr H^{n-1}}h_t(g'^{-1}\cdot g)|f(0,g' )|^qdg'.
 \end{equation}
\end{prop}

We need the following parabolic version of subharmonicity of $|u|^p$ in the Euclidean case \cite[section 3.2.1 in chapter 7]{St70}.
 \begin{prop}  \label{prop:sub}  Suppose $f$ is regular on $\mathscr U  $. Then for any  $\frac {2}{3} \leq p \leq 1$, we have
 \begin{equation}\label{eq:p-subparabolic}
   \mathcal{ L}|f|^p( t,g  )\geq 0,
 \end{equation}
 for $( t,g  )\in\mathscr U$ with $f( t,g  )\neq 0$.
\end{prop}
\begin{proof}
Since $\tau_{g_0}^*  f $ is also a regular function for fixed $g_0 \in\mathscr H^{n-1}$ and $\mathcal{ L}$ is invariant under translations, we only need to show (\ref{eq:p-subparabolic}) at point $(t,0 )$. Let $F(q)= \pi ^*f(q)$ be the pull back function on $\mathcal{U}$.
Recall that
the Cauchy--Fueter equation  $
\overline\partial_{{q}_j }F=0$ is equivalent to the  generalized Cauchy--Riemann $\overline{\partial} u(x)=0$, where  $u=(u_0,u_1)$ is a $\mathbb{C}^2$-valued function with \begin{equation}\label{eq:u-F}
      u_0=F_{ { 1}}-
 \mathbf{i}F_{ { 2}}    ,\qquad
   u_1=  F_{ { 3}}+ \mathbf{i}F_{ { 4}}
  \end{equation}
 in (\ref{eq:dbar})-(\ref{eq:GCR}). Thus $\pi^*  f $ is also a  regular function,  and by Corollary \ref{cor:subharmonic}
  \begin{equation*}
     |u|^p=|\pi^*  f |^p,
  \end{equation*}
   for   $\frac {2}{3} < p \leq 1$  is subharmonic with respect to each  quaternionic variable; i.e.
  \begin{equation}\label{eq:subharmonic-p}
     \sum_{j=4l+1}^{4l+4} \partial_{x_j}^2 |u|^p\geq 0
  \end{equation}for $l=0,\ldots,n-1$.

On the other hand, we have \begin{equation}\label{eq:p-subharmonic0}\begin{split}  \partial_t |f|^p&=\frac p2 (f,f)^{\frac p2-1}((\partial_t f,f)+( f,\partial_t f)),
  \end{split} \end{equation}
  and
    \begin{equation*}
     Y_j|f|^p =\frac p2 (f,f)^{\frac p2-1}((Y_j f,f)+(f,Y_j f)),
 \end{equation*}where $(\cdot,\cdot)$ is the quaternionic inner product on $\mathbb{H}$ given by $(q,q')= \overline{q}q'$ for $q,q'\in \mathbb{H}$. Note that $(Y_j f,f)+(f,Y_j f)=2{\rm Re} (Y_j f,f)$.
  Then, we have
   \begin{equation}\label{eq:p-subharmonic00}\begin{split}
   \sum_{j=1}^{4n-4}  Y_j^2|f|^p=&\frac p2\left(\frac p2-1\right) (f,f)^{\frac p2-2}4 \sum_{j=1}^{4n-4} \left({\rm Re} (Y_j f,f) \right)^2\\&+\frac p2 (f,f)^{\frac p2-1}\left(\bigg( \sum_{j=1}^{4n-4} Y_j^2 f ,f\bigg)+\bigg( f, \sum_{j=1}^{4n-4} Y_j^2 f\bigg)+ 2\sum_{j=1}^{4n-4}|Y_j f|^2\right).
   \end{split} \end{equation}Then
   (\ref{eq:p-subharmonic00}) minus (\ref{eq:p-subharmonic0}) multiplied by $8(n-1)$ gives us
  \begin{equation}\label{eq:p-subharmonic}\begin{split}
   8(n-1) \mathcal{ L} |f|^p=&  p \left(  p-2 \right) (f,f)^{\frac p2-2}  \sum_{j=1}^{4n-4} \big({\rm Re} (Y_j f,f) \big)^2 + p  (f,f)^{\frac p2-1}  \sum_{j=1}^{4n-4}|Y_j f|^2
   \end{split} \end{equation}by Theorem  \ref{prop:regular-sublap}.

Following the above calculation, the subharmonicity (\ref{eq:subharmonic-p}) implies
  \begin{equation}\label{eq:p-subharmonic2}\begin{split}
   0\leq  \sum_{j=1}^{4n-4} \partial_{x_j}^2 |u|^p&=   p \left(  p-2 \right) (u,u)^{\frac p2-2}  \sum_{j=1}^{4n-4} \left({\rm Re} (\partial_{x_j} u,u) \right)^2 + p  (u,u)^{\frac p2-1}  \sum_{j=1}^{4n-4}|\partial_{x_j} u|^2\\
   &= p \left(  p-2 \right) (F,F)^{\frac p2-2}  \sum_{j=1}^{4n-4} \left({\rm Re} (\partial_{x_j} F,F) \right)^2 + p  (F,F)^{\frac p2-1}  \sum_{j=1}^{4n-4}|\partial_{x_j} F|^2
 \end{split} \end{equation}
 by using (\ref{eq:u-F}). Here, by abuse of notations, $(u,u)$ or $(\partial_{x_j} u,u)$ is  complex inner product on $\mathbb{C}$.
Apply (\ref{eq:p-subharmonic2}) to   (\ref{eq:p-subharmonic})  to get $\mathcal{ L} |f|^p(t ,0 )\geq 0$ by  $Y_j f(t ,0 )=\partial_{x_j} F(t ,0 )$ by $  \pi_*\partial_{x_j}|_{(t ,0 )}=  Y_j|_{(t ,0 )} $.

The proof of Proposition \ref{prop:sub}  is complete.
  \end{proof}

 {\it Proof of Proposition  \ref{prop:heat-kernel-Hp}}.
  The
  maximum principle can not be applied to $|f|^q$  directly because it is not  smooth on $|f|=0$. It is not easy to construct an auxiliary regular function as in \cite[section 3.2.1 in chapter 7]{St70} to overcome this difficulty. But the argument of the proof of
  maximum principle can   be applied in this case as follows. Denote $v (t ,g)=|f|^q(t ,g)-\kappa t$ for $\kappa>0$ and
  \begin{equation*}
    \widehat{v} ( t, g):=\int_{\mathscr H^{n-1}}h_t(g'^{-1}\cdot g)v (0 ,g)dg'.
 \end{equation*}
 Here $v (0 ,\cdot)=|f|^q(0,\cdot)$ is in $L^r(\mathscr H^{n-1})$ with $r=p/q$. Thus, $\widehat{v} ( t, g)$ is smooth and
  \begin{equation*}
     \mathcal{L}\widehat{v} =0.
  \end{equation*}
  As (\ref{eq:estimate1})-(\ref{eq:estimate2}) in the proof of Proposition \ref{prop:heat-kernel-H1}, we can show that for given $T>0$ and $\varepsilon>0$, there exists $R>0$ such that $ | {v} -\widehat{v} |\leq \varepsilon$ on $ [0,T) \times\partial B(0,R) $ with  the mean value formula replaced by the submean value inequality for $|\pi^*\circ\tau_{g_0}^* f|^q$ in (\ref{eq:estimate2}). Then,
 \begin{equation}\label{eq:L-+}
    \mathcal{L}({v} -\widehat{v} -2\varepsilon)(t,g )>0
 \end{equation} when $|{v} | (t,g )\neq 0$,
  by  Proposition \ref{prop:sub} and $\mathcal{L} (-\kappa t) =\kappa>0$. Moreover,  ${v} -\widehat{v} -2\varepsilon$ is negative on $ [0,T) \times\partial B(0,R) \cup \{0\}\times B(0,R)$.

  Now suppose that $({v} -\widehat{v} -2\varepsilon)(t,g  )\geq 0$ at some point in  $ (0,T) \times B(0,R)$. Then, as in the proof of Proposition \ref{prop:max}, we can find $(t^*,g^* ) \in  (0,T) \times B(0,R)$
such that $({v} -\widehat{v} -2\varepsilon)(t^*,g^* )=0$  and $({v} -\widehat{v} -2\varepsilon)(t,g  )<0$ for $0<t<t^*$, $g\in B(0,R)$. Note that   we must have $|{v} |(t^*,g^* )\neq0 $. Otherwise, we have $({v} -\widehat{v} -2\varepsilon)(t^*,g^* )<0$, since $\widehat{v} ( t, g) =\int_{\mathscr H^{n-1}}h_t(g'^{-1}\cdot g) |f|^q(0,\cdot)dg' >0$ on $\mathscr U$ by positivity of the heat kernel \cite[Theorem IV 4.3]{VSC}. Consequently, ${v} -\widehat{v} -2\varepsilon$ is smooth at $(t^*,g^* )$ and so we have $\mathcal{L}({v} -\widehat{v} -2\varepsilon)(t^*,g^* )\leq0$, which contradicts (\ref{eq:L-+}). Thus $({v} -\widehat{v} -2\varepsilon)(t,g  )<0$ on  $ (0,T) \times B(0,R)$ for any fixed $\varepsilon, T >0$ and sufficiently large $R$. The result follows.

The proof of Proposition  \ref{prop:heat-kernel-Hp} is complete.\qed

\section{Hardy space $H^p(\mathscr U)$ and the regular atomic decomposition }

We need some results about the Hardy space $H^p$ over homogeneous groups in {FoSt}, of which    the quaternionic Heisenberg group is a special case. Recall that if $f\in \mathcal{S}'(\mathscr H^{n-1})$ and $\phi\in  \mathcal{S}$, the {\it nontangential  maximal function } $M_\phi f$ is defined as
\begin{equation*}
   M_{\phi}f(g):=\sup\limits_{\|g'^{-1}\cdot g\|  <t } \big|f* \phi_t(g')\big|.
\end{equation*}
If $N\in\mathbb{ N}$, $f\in \mathcal{S}'(\mathscr H^{n-1})$, we define the {\it nontangential  grand maximal function } $M_{(N )}f$ by
\begin{equation*}
   M_{(N )}f(g) := \sup\limits_{\phi\in  \mathcal{S}, \|\phi\|_{(N)}\leq1} M_{\phi}f(g),
\end{equation*}
where
\begin{equation*}
   \|\phi\|_{(N)}:=\sup\limits_{ |I|  \leq N, g\in \mathscr H^{n-1}} (1+\|g\|)^{(N+1)(Q+1)}\big|X^I \phi (g)\big|,
\end{equation*} $Q=4n+2$.
If $0<p\leq 1$, define the {\it (boundary) Hardy space} $H^p (\mathscr H^{n-1})$ over the quaternionic Heisenberg group  to be
\begin{equation*}
   H^p (\mathscr H^{n-1}):=\left\{f \in \mathcal{S}'(\mathscr H^{n-1})\mid M_{(N_p)}f\in L^p(\mathscr H^{n-1})\right\},
\end{equation*}where $N_p:=[Q({1\over p}-1)]+1$.

If $u$ is a continuous function on $\mathscr U $, define the maximal function $u  ^*$ on $\mathscr H^{n-1}$ by
 \begin{equation}\label{eq:nontangential}
    u  ^*(g) := \sup\limits_{\|g'^{-1}\cdot  g\|^2 <t } \big|u(t, g')\big|.
 \end{equation}
\begin{prop} \label{prop:FS} \cite[Proposition 8.4]{FoSt}
Suppose that $0<p\leq 1$, and $f\in \mathcal{S}'(\mathscr H^{n-1})$. Then $f\in H^p (\mathscr H^{n-1})$ if and only if $ u  ^*\in L^p(\mathscr H^{n-1})$, where $u(g)=f* h_t(g)$.
 \end{prop}
In Proposition \ref{prop:FS}, $ h_tf(g)$ is well defined since the heat kernel $h_t \in \mathcal{S}(\mathscr H^{n-1})$ {\cite[Proposition 1.74]{FoSt}.}
\begin{thm}\cite{FoSt}\label{lem-atom}
Suppose $0<p\leq 1$, then $H^p_{at}(\mathscr H^{n-1})= H^p(\mathscr H^{n-1} )$ and  $\|\cdot\|_{H^p_{at}(\mathscr H^{n-1})}\approx \|\cdot\|_{H^p(\mathscr H^{n-1})}$.
\end{thm}

 \begin{prop} \label{prop:bded}
 If $f\in H^p(\mathscr U )$ for $\frac{2}{3}<p\leq 1$, then $f(\varepsilon,\cdot)\in H^p (\mathscr H^{n-1})$ for any $\varepsilon>0$ and
 $$\|f(\varepsilon,\cdot)\|_{H^p (\mathscr H^{n-1})}\leq C\|f\|_{H^p(\mathscr U )}.$$
 \end{prop}

 \begin{proof}
 Note that $f$ is smooth on $\mathscr U$ since $\pi^*f$ is harmonic.   For fixed $\varepsilon>0$, $f(\varepsilon+\cdot,\cdot)\in H^p(\mathscr U )$ by definition and is continuous on $\overline{\mathscr U}$. Apply     Proposition \ref {prop:heat-kernel-Hp} to it to get
 \begin{equation*}
 |f(t+\varepsilon,g)|^q\leq\int_{\mathscr H^{n-1}}h_t(g'^{-1}\cdot g) |f(\varepsilon,g')|^qdg'
 \end{equation*} if we choose $\frac{2}{3}<q<p$.
 Then,  $r=\frac{p}{q}>1$ and $|f(\varepsilon,\cdot)|^q\in L^r(\mathscr H^{n-1})$ and
 \begin{equation*}
 \sup_{|g^{-1}\cdot h|<t}|f(t+\varepsilon,g)|^q\leq  \sup_{|g^{-1}\cdot h|<t}\int_{\mathscr H^{n-1}}h_t(g'^{-1}\cdot g) |f(\varepsilon,g')|^qdg'\leq
 { C_0M\left({|f(\varepsilon,\cdot)|^q}\right)(h)},
 \end{equation*}
where $M$ is the Hardy--Littlewood maximal function on $\mathscr H^{n-1}$.

On the other hand,  $f(t+\varepsilon,g) =e^{-\frac{t}{8(n-1)}\Delta_H}f(\varepsilon,g)$
by Proposition \ref{prop:heat-kernel-H1} since we can see that $f(\varepsilon+\cdot,\cdot) \in H^1(\mathscr U )$ by its boundedness  on $\mathscr U  $ by Lemma \ref {lem:bound}.
Hence,
 \begin{align*}
 \left\|  \sup_{{|g^{-1}\cdot h|<t}}e^{-\frac{t}{8(n-1)}\Delta_H}f(\varepsilon,g) \right\|_{L^p(\mathscr H^{n-1})}
& =\left\|\sup_{{|g^{-1}\cdot h|<t}} |f(t+\varepsilon,g)|^q\right\|_{L^r(\mathscr H^{n-1})}
\leq   C_0\left\|M\left({|f(\varepsilon,\cdot)|^q}\right)(h)\right\|_{L^r(\mathscr H^{n-1})}  \\
&\leq C C_0\|f(\varepsilon,\cdot)\|_{L^p(\mathscr H^{n-1})}\leq C C_0\|f\|_{H^p(\mathscr U )}.
 \end{align*}
 Thus, $f(\varepsilon,\cdot)\in H^p (\mathscr H^{n-1}) $ by Proposition \ref{prop:FS} and their $H^p (\mathscr H^{n-1}) $ quasinorms are uniformly bounded.

The proof of Proposition \ref{prop:bded} is complete.
 \end{proof}

Recall that a quasinorm $\|\cdot\|_A$ of vector space $A$ is called \emph {$c$-norm} if
 $\|a+b\|_A\leq c\left(\|a\|_A+\|b\|_A\right).$

 \begin{lem}[\cite{BL}, [Lemma 3.10.1]\label{lem:cnorm}
 For a $c$-norm on vector space $A$, there exists a norm $\|\cdot\|^*_A$ on $A$ such that
 \begin{align}\label{cnorm}
 \|\cdot\|^*_A\leq \|\cdot\|^\rho_A\leq 2\|\cdot\|^*_A,
 \end{align}
 with $(2c)^\rho=2$.
 \end{lem}

We now turn to the proof of Theorem \ref{thm2}.

\begin{proof}[Proof of Theorem \ref{thm2}]
(1) Suppose $f\in H^p(  \mathscr U )$. We now prove that $f\in H_{at} ^p(  \mathscr U )$.

To begin with,
for $\mathfrak f \in L^1(\mathscr H^{n-1})\cap L^\infty (\mathscr H^{n-1})$ and fixed $(t,g)\in\mathscr U$, let
\begin{equation*}
   S_{(t,g)}(\mathfrak f ):=\mathcal P(\mathfrak f ) (t,g)
\end{equation*}
  be defined as the integral of the product of the bounded function $
 K((t,g), \cdot)
$ and $\mathfrak f $.  Thus,   $S_{(t,g)}$ is a linear functional on $H ^p(\mathscr H^{n-1})\cap L^1(\mathscr H^{n-1})\cap L^\infty(\mathscr H^{n-1})$,  a dense subspace of the quasi-Banach space $ H^p(\mathscr H^{n-1})$ by \cite[Theorem 3.33] {FoSt}. We claim that $S_{(t,g)}$ is continuous on $H ^p(\mathscr H^{n-1})\cap L^1(\mathscr H^{n-1})\cap L^\infty(\mathscr H^{n-1})$ with respect to the $H ^p(\mathscr H^{n-1})$  norm. Then it can be naturally extended to a continuous linear functional on $H^p(\mathscr H^{n-1})$.

To prove the claim for $\mathfrak f =\sum_{\alpha=1}^4 \mathfrak f _\alpha \mathbf{i}_\alpha\in H ^p(\mathscr H^{n-1})\cap L^1(\mathscr H^{n-1})\cap L^\infty(\mathscr H^{n-1})$ ($\mathbf{i}_1=1, \mathbf{i}_2=\mathbf{i},\mathbf{i}_3=\mathbf{j},\mathbf{i}_4=\mathbf{k}$), note that scalar functions $\mathfrak f _\alpha \in H ^p(\mathscr H^{n-1})\cap L^1(\mathscr H^{n-1})\cap L^\infty(\mathscr H^{n-1})$ has the Calder\'on--Zygmund decomposition (\cite[Section B in chapter 3] {FoSt})
 $$\mathfrak f _\alpha = g_\alpha ^k+\sum_i b^k_{\alpha;i}$$  of degree $a$ and height $2^k$, where $b^k_{\alpha;i}$ is supported on $B(g_j,2r_j)$ with
\begin{equation*}
   \cup B(g_j,r_j)=\Omega_\alpha^k:=\{g\in \mathscr H^{n-1}; M_{(N)}f_\alpha(g)>2^k\}  , \qquad B(g_j,T_2 r_j)\cap (\Omega_\alpha^k)^c=\varnothing
\end{equation*}
  and no point of $\Omega_\alpha^k$ lies in more than $L$ of the balls $B(g_j,T_2 r_j)$.

Note also that by \cite[Theorem 3.17] {FoSt}, $ g_\alpha ^k\to \mathfrak f _\alpha$ in  $H ^p(\mathscr H^{n-1})$ as $k\to +\infty$ and by \cite[Theorem 3.20] {FoSt} $ g_\alpha ^k\to0$ uniformly as $k\to -\infty$.
Hence, $\mathfrak f _\alpha $
has the following decomposition  \cite[Section B in chapter 3] {FoSt}
\begin{equation}\label{eq:sum-g}
   \mathfrak f _\alpha =\sum_{k=-\infty}^{\infty} (g_\alpha ^{k+1}-g_\alpha ^k).
\end{equation}

We point out that  the summation \eqref{eq:sum-g} is only taken over $k\leq N_\alpha$ for some $  N_\alpha$ depending on $ \|f_\alpha\|_{L^\infty(\mathscr H^{n-1})} $,  because $g_\alpha ^{k+1}-g_\alpha ^k=0$ for $k>N_\alpha$ by $\Omega_\alpha^k=\varnothing$ for such $k$. That is,
\begin{equation}\label{eq:sum-g cut}
   \mathfrak f _\alpha =\sum_{k=-\infty}^{N_\alpha} (g_\alpha ^{k+1}-g_\alpha ^k).
\end{equation}
Moreover, we have that
\begin{equation}\label{eq:sum-g1}
   g_\alpha ^{k+1}-g_\alpha ^k=\sum_i \lambda^k_{\alpha;i} a^k_{\alpha;i}
\end{equation}
with $a^k_{\alpha;i}$ being a $(p,\infty,\alpha)$-atom supported on $B(g_i,T_2 r_i)$, where $\lambda^k_{\alpha;i}\in\mathbb R$ and
\begin{equation}\label{eq:estimate-atom}
\|\lambda^k_{\alpha;i} a^k_{\alpha;i}\|_{L^\infty(\mathscr H^{n-1})}\leq C_2 2^k,\qquad
   \sum_k\sum_i|\lambda^k_{\alpha;i} |^p\leq C_4\|Mf_\alpha\|_{L^p(\mathscr H^{n-1})}^p.
\end{equation}
for some absolute constants $C_2,C_4>0$.

Note that
\begin{align}\label{eq:sum-g1}
  \sum_{k=-\infty}^{N_\alpha}  \sum_i \|\lambda^k_{\alpha;i} a^k_{\alpha;i}\|^2_{L^2(\mathscr H^{n-1})}&\leq\sum_{k=-\infty}^{N_\alpha}  \sum_i C_2 2^{2k}|B(g_j,T_2 r_j)|
  \leq C_2 L\sum_{k=-\infty}^{N_\alpha}    2^{2k}|\Omega_\alpha^k |\\
  &\leq 2^{ N_\alpha(2-p )}C_2  L\sum_{k=-\infty}^{N_\alpha}    2^{ k(p-1)}|\Omega_\alpha^k |2^{k-1}\nonumber\\
  &\leq  2^{ N_\alpha(2-p )}C_2  L\int_{0}^{+\infty }s^{p-1}\left|\left\{g\mid M_{(N)}f_\alpha(g)>s\right\}\right|ds\nonumber\\
  &=2^{ N_\alpha(2-p )}C_2  L\|M_{(N)}f_\alpha\|_{L^p(\mathscr H^{n-1})}^p\nonumber\\
  &\approx
  2^{ N_\alpha(2-p )}C_2  L\| f_\alpha\|_{H ^p(\mathscr H^{n-1})}^p\nonumber\\
  & <\infty,\nonumber
\end{align}
by $f_\alpha \in H ^p(\mathscr H^{n-1}) \cap L^1(\mathscr H^{n-1}) \cap L^\infty(\mathscr H^{n-1}) $.
 Thus, we have
 \begin{equation*}
    \sum_{\alpha=1}^4\sum_{k=-\infty}^{N_\alpha}  \sum_i  a^k_{\alpha;i}\lambda^k_{\alpha;i}\mathbf{i}_\alpha= f
 \end{equation*}  in $L^2(\mathscr H^{n-1})$, and by the continuity of $S_{(t,g)}$ on  $L^2(\mathscr H^{n-1})$, we find that
 \begin{equation*}\begin{split}
    S_{(t,g)}(f) &= \sum_{\alpha } \sum_{k=-\infty}^{N_\alpha}  \sum_i   S_{(t,g)}(a^k_{\alpha;i})\lambda^k_{\alpha;i}\mathbf{i}_\alpha
          \end{split}\end{equation*}for $f \in H ^p(\mathscr H^{n-1})\cap L^1(\mathscr H^{n-1})\cap L^\infty(\mathscr H^{n-1})$.
               Then we can apply  this identity   to get
       \begin{align*}
   \left|S_{(t,g)}(f)\right|&\leq
  \sum_{\alpha } \sum_{k=-\infty}^{N_\alpha}  \sum_i \left|\mathcal P(a^k_{\alpha;i})(t,g) \lambda^k_{\alpha;i}\right|\leq C\sum_{\alpha } \sum_{k=-\infty}^{N_\alpha}  \sum_i \|\mathcal P(a^k_{\alpha;i})\|_{H^p(\mathscr U)}t^{-\frac{2n+1}{p}} \left|  \lambda^k_{\alpha;i}\right|\\
   &\leq CC_{p,n,\alpha}t^{-\frac{2n+1}{p}}\left(\sum_{\alpha } \sum_{k=-\infty}^{N_\alpha}  \sum_i \left|  \lambda^k_{\alpha;i}\right|^p\right)^{\frac 1p}\\
   &\approx CC_{p,n,\alpha}t^{-\frac{2n+1}{p}}\| f\|_{H ^p(\mathscr H^{n-1}) },
\end{align*}
by using Lemma \ref {lem:bound}, Proposition \ref {prop0}  and \eqref{eq:estimate-atom}.
  Thus $S_{(t,g)}$ is bounded on $H ^p(\mathscr H^{n-1}) \cap L^1(\mathscr H^{n-1}) \cap L^\infty(\mathscr H^{n-1}) $ with respect to the $H ^p(\mathscr H^{n-1})$  norm.

We now apply Lemma \ref{lem:cnorm} to the  quasi-Banach space $H ^p(\mathscr H^{n-1})$ to get   a norm $\|\cdot\|^*_{H ^p(\mathscr H^{n-1})}$. Then by \eqref{cnorm},  it is obvious that $\big(H ^p(\mathscr H^{n-1}), \|\cdot\|^*_{H ^p(\mathscr H^{n-1})} \big)$ is complete. Namely, $\big(H ^p(\mathscr H^{n-1}), \|\cdot\|^*_{H ^p(\mathscr H^{n-1})} \big)$ is a Banach space.

By Proposition \ref{prop:bded}, we see that $\{f(\varepsilon,\cdot) \}_{\varepsilon>0}$ is a bounded set in the quasi-Banach space $H ^p(\mathscr H^{n-1})$ if  $f\in H^p(\mathscr U )$ for $\frac{2}{3}<p\leq 1$. It is also bounded in
$H ^p(\mathscr H^{n-1})$ under the norm $\|\cdot\|^*_{H ^p(\mathscr H^{n-1})}$. It follows from Banach--Alaoglu theorem that there exists a
 subsequence $\{f(\varepsilon_k,\cdot) \}$  weakly convergent to some $f^b$ in $\big (H ^p(\mathscr H^{n-1}), \|\cdot\|^*_{H ^p(\mathscr H^{n-1})} \big)$. Then by \eqref{cnorm} we can obtain that $\{f(\varepsilon_k,\cdot) \}$ is also weakly convergent to the same $f^b$ in the quasi-Banach space  $H ^p(\mathscr H^{n-1})$, since every continuous linear functional on the quasi-Banach space $H ^p(\mathscr H^{n-1})$ must be continuous under the norm $\|\cdot\|^*_{H ^p(\mathscr H^{n-1})}$ by \eqref{cnorm}.

 Note that $f(\varepsilon_k+t,g)$ is uniformly bounded on $\mathscr U$  by   Lemma \ref {lem:bound}, and so it is in $H^2(\mathscr U)$.
Thus we have $(\mathcal Pf(\varepsilon_k ,\cdot))(t,g)=f(\varepsilon_k+t,g)$ by the reproducing formula. Now apply $S_{(t,g)}$ to $\{f(\varepsilon_k,\cdot) \}$ to get
 \begin{align*}
 S_{(t,g)}(f^b)&=\lim_{k\to\infty}S_{(t,g)}f(\varepsilon_k,\cdot)=\lim_{k\to\infty}\mathcal P\big(f(\varepsilon_k,\cdot)\big)(t,g)
 =\lim_{k\to\infty}f(\varepsilon_k+t,g)
 =f(t,g).
 \end{align*}
On the other hand, since $f^b$ in $\big (H ^p(\mathscr H^{n-1}), \|\cdot\|^*_{H  ^p(\mathscr H^{n-1})} \big)$, we have the atomic decomposition $f^b
 =\sum_k  a_k \lambda_k $ by Theorem \ref{lem-atom}, and
 $$ S_{(t,g)}(f^b) 
 =\sum_k \mathcal P(a_k)(t,g)\lambda_k,$$ 
 by continuity of $S_{(t,g)}$ on $H ^p(\mathscr H^{n-1})$.
Consequently,  $f(t,g)=\sum\limits_k \mathcal P(a_k)(t,g)\lambda_k$ for each point $ (t,g)\in\mathscr U $, i.e., $f$ is in $   H_{at} ^p(\mathscr U)$.

 \color{black}
 \medskip

(2) Suppose $u\in H_{at} ^p(  \mathscr U )$. We now prove that $u\in H^p(  \mathscr U )$.

  Let $u=\sum_{j=1}^\infty A_j{ {\lambda_j}} \in H_{at} ^p(\mathscr U )$ such that $\sum_{j=1}^\infty |\lambda_j|^p\approx \|u\|_{H_{at} ^p(\mathscr U )}$, where $A_j    $'s are regular $p$-atoms; i.e.
there exist $(p,\infty,\alpha)$-atoms $a_j$ on $\mathscr H^{n-1}$  such that
    $A_j = \mathcal P(a_j)  $.
    By  Proposition \ref{prop0},
we see that $A_j\in H^p(\mathscr U )$ with $\|A_j\|_{H^p(\mathscr U )}\leq C_{p,n,\alpha}$ for all $j$.
  Apply Proposition \ref {lem:bound} to $A_j$ to see that $ |A_j(\varepsilon,g )|\leq  CC_{p,n,\alpha}\varepsilon^{-\frac {2n+1}p}$, and so $\sum_{j=1}^\infty A_j\lambda_j $
  converges uniformly on compact subset of $\mathscr U$ and defines a regular function on $\mathscr U$ and
  \begin{equation*}
   \left | \sum_{j=1}^\infty A_j( \varepsilon,g)\lambda_j \right|\leq  CC_{p,n,\alpha}\varepsilon^{-\frac {2n+1}p}\sum_{j=1}^\infty |\lambda_j|.
  \end{equation*}
  Moreover,
  \begin{align*}
     \int_{\mathscr H^{n-1} }\left | \sum_{j=1}^\infty A_j( \varepsilon,g)\lambda_j\right|^pdg&\leq  \int_{\mathscr H^{n-1} } \sum_{j=1}^\infty   | \lambda_j|^p| A_j( \varepsilon,g) |^pdg
     \leq \sum_{j=1}^\infty   | \lambda_j|^p \int_{\mathscr H^{n-1} } |A_j( \varepsilon,g) |^pdg\\&
      \leq C_{p,n,\alpha}^p\sum_{j=1}^\infty   | \lambda_j|^p \\
      & \approx C_{p,n,\alpha}^p\|u\|_{H_{at}^p( \mathscr U )}.
\end{align*}
Namely, $u=\sum_{j=1}^\infty A_j\lambda_j \in  H^p( \mathscr U ) $ with
$
   \|u\|_{H^p(  \mathscr U )}\leq C_p \|u\|_{H_{at} ^p( \mathscr U )}
$.

The proof of Theorem \ref{thm2} is complete.
  \end{proof}

\section{Singular value estimates of the commutator $[b,\mathcal P]$: proof of Theorem \ref{schatten}}

Based on Theorems \ref{main1} and \ref{main3}, and the recent result in \cite{FLL}, we see that (1) in Theorem \ref{schatten} holds.\\

We now prove (2) in Theorem \ref{schatten}.
The sufficient condition is obvious, since $[b,\mathcal P]=0$ when $b$ is a constant. We now prove necessary condition. To show this,
 it suffices to consider  the critical case $p=4n+2$, by the inclusion $ S^p\subset S^{2n+2}$ for $p<2n+2$.

\begin{lem}\label{signlemma}
There exists a positive integer $\mathfrak b$ such that for any tile $T\in \tile_{k}$ and $a_{j}=\pm 1$ ($j=1,2,\cdots,4n-4$), there are tiles $T^{\prime}\in\tile_{k-\mathfrak b}$, $T^{\prime\prime}\in\tile_{k-\mathfrak b}$ such that $T^{\prime}\subset T$,
$T^{\prime\prime}\subset T$ and if $g=(g_{1},\cdots,g_{4n-4},\bm t)\in T^{\prime\prime}$, $h=(h_{1},\cdots,h_{4n-4},\bm t^{\prime})\in T^{\prime}$, then $a_{j}(g_{j}-h_{j})\gtrsim {\rm width}(T)$ $(j=1,2,\ldots,4n-4)$.
\end{lem}
\begin{proof}
Consider first $T = \delta_{2^{k} }(A)$.  Based on (4) in Lemma \ref{thm:Heisenberg-grid}, there exist $o=(\bm x, \bm t)\in\mathscr H^{n-1}$ such that  $B(o, C_1 2^{k})\subset T$. Then one can choose $g_{o,1}=(\tilde{\bm x},\tilde{\bm{ t}})\in B(o, C_12^{k})$ such that $d(g_{o,1},o) = {3C_1\over4}2^{k}$, and
$$\tilde {x}_j-x_j={3C_1\over 4\sqrt{4n-4}} 2^{k},\quad j=1,\cdots, 4n-4.$$
Thus, we can choose $C'_1$ sufficiently small such that for each $g=(\bm{x}', \bm{t}')\in  B(g_{o,1}, C'_12^{k})$, we have
$$x'_j-x_j>{3C_1\over 8\sqrt{4n-4}} 2^{k},\quad j=1,\cdots, 4n-4.$$

Take $g_{o,2}=o(o^{-1}g_{o,1})^{-1}$ and write $g_{o,2}=(\hat{\bm x}, \hat{\bm t})$, then
$d(g_{o,2}, o)={3C_1\over4}2^{k}$ and
$$\hat {x}_j-x_j=-{3C_1\over 4\sqrt{4n-4}} 2^{k},\quad j=1,\cdots, 4n-4.$$
For any $\tilde g=(\bm y, \bm s)\in B(g_{o,1}, C'_12^{k})$, we have
$$y_j-x_j<-{3C_1\over 8\sqrt{4n-4}} 2^{k},\quad j=1,\cdots, 4n-4.$$

As a consequence, if we choose $\mathfrak b$ sufficiently large such that $2^{k-\mathfrak b}{\rm width}(A)<C'_12^{k}$, then there must exist $T^{\prime}\in\tile_{k-\mathfrak b}$ such that $T'\subset B(g_{o,1}, C'_12^{k})$ and $T^{\prime\prime}\in\tile_{k-\mathfrak b}$ such that
$T''\subset B(g_{o,2}, C'_12^{k})$. Then it is clear that if $g\in T^{\prime}$, $h\in T^{\prime\prime}$, then $g_{j}-h_{j}\gtrsim {\rm width}(T)$ $(j=1,2,\ldots,4n-4)$. The proof of other cases are similar.
This ends the proof of Lemma \ref{signlemma}.
\end{proof}

Recall the following first order Taylor's inequality on $\mathscr H^{n-1}$ from \cite{Taylorformula}.
\begin{lem}\label{taylor}
Let $f\in C^{\infty}(\mathscr H^{n-1})$, then for every $g=(x_{1},\cdots,x_{4n-4},\bm t),g_{0}=(x_{0}^{1},\ldots,x_{0}^{4n-4}, \bm t_0)\in \mathscr H^{n-1}$, we have
\begin{align*}
f(g)=f(g_{0})+\sum_{k=1}^{4n-4}\frac{Y_{k}f(g_{0})}{k!}(x_{k}-x_{0}^{k})+R(g,g_{0}),
\end{align*}
where the remainder  $R(g,g_{0})$ satisfies the following inequality:
\begin{align*}
|R(g,g_{0})|\leq C\left(\sum_{k=1}^{2}\frac{c^{k}}{k!}\sum_{\substack{i_{1},\ldots,i_{k}\leq 4n-1,\\ I=(i_{1},\ldots,i_{k}),\ d(I)\geq 2}}\|g_{0}^{-1}\cdot g\|^{d(I)}\sup\limits_{\|h\|\leq c\|g_{0}^{-1}\cdot g\|}|Y^{I}f(g_{0}z)|\right)
\end{align*}
for some constant $c>0$.
\end{lem}

In the sequel, for any $T\in\tile_{k}$, let $T^{\prime}$ be the tile chosen in Lemma \ref{signlemma}. Also, we denote $\nabla_H$ be the horizontal gradient of $\mathscr H^{n-1}$ defined by $\nabla_H f:=(Y_{1}f,\cdots,Y_{4n-4}f)$. Then we can show a lower bound for a local pseudo-oscillation of the symbol $b$ in the commutator.
\begin{lem}\label{lowerbound}
Let $b\in C^{\infty}(\mathscr H^{n-1})$. Assume that there is a point $g_{0}\in\mathscr H^{n-1}$ such that $\nabla b(g_{0})\neq 0$. Then there exist $C>0$, $\varepsilon>0$ and $N>0$ such that if $k>N$, then for any tile $T\in \tile_{k}$ satisfying $d(\cent(T),g_{0})<\varepsilon$, one has
\begin{align*}
\frac{1}{|T|}\int_{T}|b(g)-(b)_{T^{\prime}}|\,dg\geq C{\rm width}(T)|\nabla_H b(g_{0})|.
\end{align*}
\end{lem}
\begin{proof}
Denote $c_{T}:=\cent(T):=\{c_{T}^{1},\ldots,c_{T}^{4n-4},\bm t_T\}$ and $g=(g_1,\cdots,g_{4n-4},\bm t)$, then by Lemma \ref{taylor},
$$b(g)=b(c_{T})+\sum_{j=1}^{4n-4}\frac{Y_{j}b(c_{T})}{j!}(g_{j}-c_{T}^{j})+R(g,c_{T}),$$
where the remainder term $R(g,c_{T})$ satisfies
\begin{align*}
|R(g,c_{T})|\leq C\left(\sum_{j=1}^{2}\frac{c^{j}}{j!}\sum_{\substack{i_{1},\ldots,i_{j}\leq 4n-1,\\ I=(i_{1},\ldots,i_{j}),\ d(I)\geq 2}}\|c_{T}^{-1}\cdot g\|^{d(I)}\sup\limits_{\|h\|\leq c\|c_{T}^{-1}\cdot g\|}|Y^{I}b(c_{T}\cdot h)|\right).
\end{align*}
Note that the condition $\|h\|\leq c\|c_{T}^{-1}\cdot g\|$ implies that $d(c_{T}\cdot h,c_{T})=\|h\|\leq c\|c_{T}^{-1}\cdot g\|\lesssim {\rm width}(T)$ whenever $g\in T$. Hence, if $g\in T$, then
\begin{align*}
|R(g,c_{T})|\leq C{\rm width}(T)^{2}\sum_{j=1}^{2}\sum_{\substack{i_{1},\ldots,i_{j}\leq 4n-1,\\ I=(i_{1},\ldots,i_{j}),\ d(I)\geq 2}}\|Y^{I}b\|_{L^{\infty}(B(g_{0},1))}.
\end{align*}
Besides, it follows from Lemma \ref{signlemma} that there exist  cubes $T^{\prime}\in\tile_{k-\mathfrak b}$, $T^{\prime\prime}\in\tile_{k-\mathfrak b}$ such that $T^{\prime}\subset T$,
$T^{\prime\prime}\subset T$ and ${\rm sgn}(Y_{j}b)(c_{T})(g_{j}-\tilde g_{j})\gtrsim {\rm width}(T)$ ($j=1,2,\ldots,4n-4$). Therefore,
\begin{align*}
&\frac{1}{|T|}\int_{T}|b(g)-(b)_{T^{\prime}}|dg\\
&\geq \frac{1}{|T| |T'|}\int_{T}\left|\int_{T^{\prime}}\sum_{j=1}^{4n-4}\frac{(Y_{j}b)(c_{T})}{j!}(g_{j}-\tilde g_{j})dh\right|dg-{1\over |T|}\int_{T}|R(g,c_{T})|dg-{1\over |T'|}\int_{T^{\prime}}|R(\tilde g,c_{T})|dh \\
&\geq C\sum_{j=1}^{4n-4}|Y_{j}b(c_{T})|{\rm width}(T)-C{\rm width}(T)^{2}\sum_{j=1}^{2}\sum_{\substack{i_{1},\ldots,i_{j}\leq 4n-1,\\ I=(i_{1},\ldots,i_{j}),\ d(I)\geq 2}}\|Y^{I}b\|_{L^{\infty}(B(g_{0},1))}\\
&\geq C{\rm width}(T)|\nabla_H b(g_{0})|,
\end{align*}
where the last inequality holds since we choose $N$ to be a sufficient large constant such that the remainder term can be absorbed by the first term.
\end{proof}

 {Denote by $\tau^h$ the right translation by $h$.}
\begin{lem}\label{const}
Let $b\in L_{{\rm loc}}^{1}(\mathscr H^{n-1})$. Suppose that
\begin{align*}
\sup\limits_{h\in B(0,1)}\left\|\left\{\frac{1}{|\tau^{h}T|}\int_{\tau^{h}T}|b(g)-(b)_{\tau^{h}T^{\prime}}|dg\right\}_{T\in\tile}\right\|_{\ell^{4n+2}}<+\infty,
\end{align*}
then $b$ is a constant.
\end{lem}
\begin{proof}
Denote $\psi_{\epsilon}(g):=\epsilon^{-4n-2}\psi(\delta_{\epsilon^{-1}}g)$, where $\psi$ is a smooth bump function and $\epsilon$ is a small positive constant.
Note that
\begin{align}\label{contra}
&\left\|\left\{\frac{1}{|T|}\int_{T}|b\ast \psi_{\epsilon}(g)-(b\ast\psi_{\epsilon})_{T^{\prime}}|dg\right\}_{T\in\tile}\right\|_{\ell^{4n+2}}\\
&\leq \left\|\left\{\int_{\mathscr H^{n-1}}|\psi_{\epsilon}(h^{-1})|\frac{1}{|\tau^{h}T|}\int_{\tau^{h}T}|b(g)-(b)_{\tau^{h}T^{\prime}}|dgdh\right\}_{T\in\tile}\right\|_{\ell^{4n+2}}\nonumber\\
&\leq \sup\limits_{h\in B(0,1)}\left\|\left\{\frac{1}{|\tau^{h}T|}\int_{\tau^{h}T}|b(g)-(b)_{\tau^{h}T^{\prime}}|dg\right\}_{T\in\tile}\right\|_{\ell^{4n+2}}\nonumber\\
&<+\infty.\nonumber
\end{align}

Next we will show that for any $\epsilon>0$, $b\ast\psi_{\epsilon}$ is a constant, which implies that $b$ is a constant by letting $\epsilon\rightarrow 0$. If not,  then it follows from \cite[Proposition 1.5.6]{BLU} that there exists a point $g_{0}\in \mathscr H^{n-1}$ such that $\nabla_H b\ast\psi_{\epsilon}(g_{0})\neq 0$. By Lemma \ref{lowerbound}, there exist  $\varepsilon>0$ and $N>0$ such that if $k>N$, then for any tile $T\in \tile_{k}$ satisfying $d(\cent(T),g_{0})<\varepsilon$,
\begin{align*}
\frac{1}{|T|}\int_{T}|b\ast \psi_{\epsilon}(g)-(b\ast \psi_{\epsilon})_{T^{\prime}}|dg\geq C{\rm width}(T)|\nabla b\ast \psi_{\epsilon}(g_{0})|.
\end{align*}
Note that for $k>N$, the number of $T\in\tile_{k}$ and $d(\cent(T),g_{0})<\varepsilon$ is at least  {$c 2^{k(4n+2)}$}. Therefore,
\begin{align*}
&\left\|\left\{\frac{1}{|T|}\int_{T}|b\ast \psi_{\epsilon}(g)-(b\ast\psi_{\epsilon})_{T^{\prime}}|dg\right\}_{T\in\tile}\right\|_{\ell^{4n+2}}\\
&\gtrsim  \Bigg(\sum_{k=N+1}^{\infty}\sum_{\substack{T\in\tile_{k}\\ T:d(\cent(T),g_{0})<\varepsilon}}{\rm width}(T)^{4n+2}|\nabla b\ast \psi_{\epsilon}(g_{0})|^{4n+2}\Bigg)^{1\over 4n+2}\\
&\gtrsim |\nabla_H b\ast \psi_{\epsilon}(g_{0})|\Bigg(\sum_{k=N+1}^{\infty}\sum_{\substack{T\in\tile_{k}\\ T:d(\cent(T),g_{0})<\varepsilon}} 2^{k (4n+2)}\Bigg)^{1\over 4n+2}\\
&=+\infty.
\end{align*}
This is in contradiction with inequality \eqref{contra}. Therefore, the proof of Lemma \ref{const} is complete.
\end{proof}

Based on our fundamental results in Lemmas \ref{signlemma}---\ref{const},
(2) holds by using the argument in \cite{FLL}. Hence, the proof  of Theorem \ref{schatten} is complete. \qed

  \bigskip

  \bigskip

{\bf Acknowledgement:} 
Chang is supported by NSF grant DMS-1408839 and a McDevitt
Endowment Fund at Georgetown University. Duong and Li are supported by the Australian
Research Council (ARC) through the research grants DP 190100970 and DP 170101060. Wang is supported by National Nature
Science Foundation in China (NNSF)  (No. 11971425). Wu is supported by NNSF (No. 12171221 and No. 12071197), the NSF of
Shandong Province (No. ZR2018LA002 and No. ZR2019YQ04).

\color{black}

\end{document}